\documentclass[reqno, 11pt, a4paper]{amsart}
\usepackage{a4wide,
            centernot,
            bbm,
            bm,
            centernot,
            amssymb,
            mathtools,
            xcolor,
            graphicx,
            enumitem,comment} 

\usepackage[colorlinks=false, linkcolor=blue, citecolor=blue]{hyperref}
\newtheorem{theorem}{Theorem}[section]
\newtheorem{lemma}[theorem]{Lemma}
\newtheorem{prop}[theorem]{Proposition}

\newtheorem{corollary}[theorem]{Corollary}

\theoremstyle{definition}

\theoremstyle{remark}
\newtheorem{remark}[theorem]{Remark}

\numberwithin{equation}{section}

\newcommand{\IND}{\mathbbm{1}}

\newcommand{\De}{\mathrm{d}}
\newcommand{\rmR}{\ensuremath{\mathrm{R}}}
\newcommand{\rmP}{\ensuremath{\mathrm{P}}}

\newcommand{\cC}{\ensuremath{\mathcal C}}

\newcommand{\cH}{\ensuremath{\mathcal H}}
\newcommand{\cI}{\ensuremath{\mathcal I}}

\newcommand{\cP}{\ensuremath{\mathcal P}}

\newcommand{\cS}{\ensuremath{\mathcal S}}

\newcommand{\cW}{\ensuremath{\mathcal W}}

\newcommand{\bbE}{\ensuremath{\mathbb E}}

\newcommand{\bbR}{\ensuremath{\mathbb R}}

\newcommand{\frm}{\ensuremath{\mathfrak m}}

\newcommand{\frq}{\ensuremath{\mathfrak q}}

\newcommand{\R}{\mathbb{R}}
\newcommand{\RD}{\mathbb{R}^d}
\newcommand{\N}{\mathbb{N}}

\newcommand{\norm}[1]{\left\lVert #1 \right\rVert}

\newcommand{\Pexpect}[1]{\bbE_\rmP\left[#1\right]}
\newcommand{\Rexpect}[1]{\bbE_\rmR\left[#1\right]}
\newcommand{\RXexpect}[1]{\bbE_{\rmR_{0,T}}\left[#1\right]}

\newcommand{\Id}{\operatorname{Id}}
\newcommand{\abs}[1]{\left\lvert #1 \right\rvert}
\newcommand{\oo}{{\infty}}

\newcommand{\mapor}[1]{{\stackrel{#1}{\xrightarrow{\hspace*{1cm}}}}} 

\newcommand{\weakto}{\rightharpoonup} 

\newcommand{\tempo}{\frac{1}{\kappa}\log C_{d,\alpha,\beta,\gamma} +2\delta} 
\newcommand{\tempobis}{\frac{1}{\kappa}\log\frac{C_{d,\alpha,\beta,\gamma}}{\delta^{3}}}

\begin{document}

\title[Kinetic Schr\"odinger problem]{Entropic turnpike estimates for the Kinetic Schr\"odinger Problem}


\author{Alberto Chiarini}
\address{Università degli Studi di Padova}
\curraddr{Department of Mathematics ``Tullio Levi-Civita'', via Trieste 63, Padova}
\email{chiarini@math.unipd.it}
\thanks{}

\author{Giovanni Conforti}
\address{\'Ecole Polytechnique}
\curraddr{D\'epartement de Math\'ematiques Appliqu\'es,
Palaiseau, France.}
\email{giovanni.conforti@polytechnique.edu}
\thanks{GC acknowledges funding from the grant SPOT (ANR-20-CE40-
0014)}

\author{Giacomo Greco}
\address{Eindhoven University of Technology}
\curraddr{Department of Mathematics and Computer Science, 5600 MB Eindhoven}
\email{g.greco@tue.nl}
\thanks{GG acknowledges support from NWO Research Project 613.009.111 ``Analysis meets Stochastics: Scaling limits in complex systems''. The research was also partially funded by Nuffic in the framework of the Van Gogh Programme and with the title ``The kinetic Schrödinger Problem''.}

\author{Zhenjie Ren}
\address{Universit\'e Paris-Dauphine}
\curraddr{Ceremade, PSL Research University, 75016 Paris, France.}
\email{ren@ceremade.dauphine.fr}
\thanks{}

\begin{abstract}
We investigate the \emph{kinetic} Schr\"odinger problem, obtained considering Langevin dynamics instead of Brownian motion in Schr\"odinger's thought experiment. Under a quasilinearity assumption we establish exponential entropic turnpike estimates for the corresponding Schr\"odinger bridges and exponentially fast convergence of the entropic cost to the sum of the marginal entropies in the long-time regime, which provides as a corollary an entropic Talagrand inequality. In order to do so, we benefit from recent advances in the understanding of classical Schr\"odinger bridges and adaptations of Bakry--\'Emery formalism to the kinetic setting. Our quantitative results are complemented by basic structural results such as dual representation of the entropic cost and the existence of Schr\"odinger potentials.
\end{abstract}

\subjclass[2010]{}
\date{}
\dedicatory{}
\maketitle


\section{Introduction and statement of the main results}

In two seminal contributions \cite{Schr,Schr32} E.\ Schr\"odinger considered the problem of finding the most likely evolution of a cloud of independent Brownian particles conditionally to observations. This problem is nowadays known as \emph{Schr\"odinger problem} and may be viewed \cite{Mik04,leonard2012schrodinger} as a more regular and probabilistic proxy for the Optimal transport (Monge-Kantorovich) problem. This observation has motivated recent interest from both the  engineering and statistical machine learning communities \cite{chen2021stochastic,peyre2019computational}. Moreover, over the past few years, various kinds of Schr\"odinger problems have been introduced and studied in the literature with different aims and scopes such as, for example, the multiplicative Schr\"odinger problem \cite{pal2018multiplicative} and the mean field Schr\"odinger problem \cite{backhoff2020mean}. In this article we investigate the \emph{Kinetic Schr\"odinger Problem}, henceforth \ref{KSP}, with particular emphasis on the long-time and ergodic behaviour of the corresponding Schr\"odinger bridges. A heuristic formulation of \ref{KSP} is naturally given in terms of the celebrated Schr\"odinger's thought experiment. Consider a system of $N\gg1$ independent stationary particles $(X^{1}_t,\ldots,X^{N}_t)_{t\in[0,T]}$ evolving according to the Langevin dynamics
\begin{equation*}
    \begin{cases}
     \De X^i_t = V^i_t\De t,\\
    \De V^i_t = -\nabla U(X^{i}_t)\De t - \gamma V^i_t\De t +\sqrt{2\gamma}\,\De B^i_t, \quad i=1,\ldots, N,
    \end{cases}
\end{equation*}
and assume that two snapshots of the particle system at the initial time $t=0$ and at the terminal time $t=T$ have been taken. The Schr\"odinger problem is that of finding the most likely evolution of the particle system conditionally on this information. 
In order to turn this heuristic description into a sound mathematical problem, we introduce the empirical path measure 
\begin{equation*}
    \bm\mu^N := \frac1N\sum \delta_{(X^i_{\cdot},V^{i}_{\cdot})}
\end{equation*}
that is a random probability measure on the space of continuous trajectories $C([0,T];\R^{2d}):=\Omega$ and consider two probability measures $\mu,\nu$ on $\RD$, representing the observed configuration at initial and final time, that is to say
\begin{equation*}
\frac1N \sum_{i=1}^N\delta_{X^i_0} \approx \mu, \quad \frac1N\sum_{i=1}^N\delta_{X^i_T} \approx \nu .
\end{equation*}
Then, leveraging Sanov's Theorem~\cite[Theorem 6.2.10]{DemboZeitouni}, whose message is that the likelihood of a given evolution $\bm\rho$ is measured through the relative entropy
\begin{equation*}
    \mathrm{Prob}\Big[ \bm\mu^N \approx  \bm\rho \Big] \approx \exp(- N \cH(\bm \rho | \rmR ) ),
\end{equation*}
we finally arrive at the variational problem 
\begin{equation}\label{KSPd}\tag{KSPd}
 \cC_T(\mu,\nu):=\inf\, \biggl\{\cH(\rmP|\rmR): \rmP \in \cP(C([0,T];\R^{2d})), \, (X_{0})_{\#}\rmP=\mu,(X_{T})_{\#}\rmP=\nu\biggr\}.
\end{equation}
In the above, $\rmR$ is the reference probability measure, that is the law of 
 \begin{equation}\label{langevin}
    \begin{cases}
        \De X_t = V_t \De t\\
        \De V_t = -\,\nabla U(X_t)\De t - \gamma V_t\De t +\sqrt{2\gamma}\,\De B_t\,\\
        (X_0,V_0)\sim\frm,
    \end{cases}
\end{equation} 
where the invariant (probability) measure $\frm$ is given by  \begin{equation*} 
\frm(\De x,\De v) =\frac{1}{Z} e^{-U(x)-\frac{\abs{v}^2}{2}}\De x\, \De v,\end{equation*}
with $Z$ being a normalising constant.
Moreover, $(X_t,\,V_t)_{t\in [0,T]}$ denotes the canonical process on $\Omega$, $\#$ is the push-forward and $\cH(\cdot|\rmR)$ is the relative entropy functional defined on $\cP(\Omega)$ as
\[\cH(\rmP|\rmR)\coloneqq\begin{cases}
                    \Pexpect{\log\frac{\De \rmP}{\De \rmR}}&\quad\text{if }\rmP\ll\rmR,\\
                    +\oo&\quad\text{otherwise}.
                    \end{cases}\]
Given that Schr\"odinger's thought experiment is motivated by statistical mechanics and the physical relevance of the Langevin dynamics and its various applications, the study of the kinetic Schr\"odinger problem appears to be quite natural. Nevertheless, to the best of our knowledge, it seems that there has been no dedicated study so far, with the exception of~\cite{chen2015fast}. The objective of this paper is to take some steps forward in this direction, in particular by gaining a quantitative understanding of optimal solutions, called Schr\"odinger bridges.
\subsection*{Turnpike property for Schr\"odinger bridges} The turnpike property is a general principle in optimal control theory stipulating that solutions of dynamic control problems are made of three pieces: first a rapid transition from the initial state to the steady state, the \emph{turnpike}, then a long stationary phase localised around the turnpike, and finally another rapid transition to reach the final state. In order to link this concept to \ref{KSPd}, we need to rephrase it as a stochastic control problem. This task is easily accomplished thanks to classical results on the representation of path measures with finite entropy, see e.g. \cite{FOLL88,LeoGir} and we get that \ref{KSPd} is equivalent to
\begin{equation}\label{eq: KSPd stoch control}
\inf\left\{\cH((X_0,V_0)_{\#}\rmP|\frm)+\frac{1}{4\gamma}\bbE_\rmP\Big[\int_{0}^T|\alpha^{\rmP}_t|^2\De t\Big]: \rmP\in\cP(\Omega), \,\, \rmP\,\text{admissible} \right\},
\end{equation}
where a path probability measure $\rmP$ is admissible if and only if under $\rmP$, there exist a Brownian motion $(B_t)_{t\in[0,T]}$ adapted to the canonical filtration and an adapted process $(\alpha^P_t)_{t\in[0,T]}$ such that $\bbE_\rmP[\int_{0}^T|\alpha^{\rmP}|^2\De t]<+\infty$ and the canonical process satisfies

\begin{equation}\label{eq: KSP admissible P}
\begin{cases}
\De X_t = V_t\De t, \\
\De V_t = - \nabla U(X_t)\De t -\gamma V_t \De t + \alpha^P_{t}\De t + \sqrt{2\gamma}\De B_t,\\
X_0\sim \mu, X_T\sim\nu.
\end{cases}
\end{equation}
For the control problem \eqref{eq: KSPd stoch control}, the turnpike is the invariant measure $\frm$. Indeed, the natural tendency of the particle system is that of reaching configuration $\frm$ and since Schr\"odinger bridges aim at approximating as much as possible the unconditional dynamics while matching the observed configurations, they should also favour configurations close to $\frm$. Obtaining a quantitative rigorous version of this statement is one of the main objectives of this article and, in view of \eqref{eq: KSP admissible P}, it is equivalent to show that Schr\"odinger bridges satisfy the turnpike property.
 In the field of deterministic control, the turnpike phenomenon is rather well understood both in a finite and infinite dimensional setting, see either \cite{trelat2015turnpike,trelat2018steady} and references therein, or the monographs \cite{zaslavski2005turnpike,zaslavski2019turnpike}. The understanding of this phenomenon in stochastic control seems to be much more limited: see \cite{cardaliaguet2012long,cardaliaguet2013long,cardaliaguet2019long} for results on mean field games and \cite{clerc2020long,backhoff2020mean} for results on the classical and mean field Schr\"odinger problems. The reason why the turnpike property for Schr\"odinger bridges in the present context cannot be deduced from existing results lies in the hypocoercive \cite{cedric2009hypocoercivity} nature of the kinetic Fokker-Planck equation 
\begin{equation}\label{KFP}
 \partial_t f_t(x,v)= \gamma\,\Delta_v f_t(x,v)-\gamma \,v\cdot \nabla_v f_t(x,v)+\nabla U\cdot\nabla_v f_t(x,v)-v\cdot\nabla_x f_t(x,v),
\end{equation}
describing the probability density of \eqref{langevin} with respect to $\frm$. It is well known that the problem of quantifying the trend to equilibrium of this PDE is more challenging than for the classical (overdamped) Fokker-Planck equation
and this difficulty is of course reflected in the problem of establishing the turnpike property for the corresponding Schr\"odinger bridges. In this work, we rely on the important progresses made in the study of the long-time behaviour of \eqref{KFP} over the last fifteen years using either an analytical approach see e.g. \cite{Baudoin, dolbeault2015hypocoercivity,herau2004isotropic, cedric2009hypocoercivity} and references therein, or a probabilistic approach, see e.g.~\cite{eberle2019couplings, guillin2021kinetic}, as well as on the new developments around the long-time behaviour of Schr\"odinger bridges, in order to gain some understanding on \emph{controlled} versions of the kinetic Fokker-Planck equation. Leaving a more accurate comparison between our results and the existing literature to the text below, let us first present a very concise summary of our contributions and explain how this article is structured.
\subsection*{Organisation} The document is organised as follows. In the upcoming sections~\ref{sec: basic res},~\ref{sec: cost res} and~\ref{sec: tpike res} we state and comment our main results. In particular, Section~\ref{sec: basic res} contains additional background material on the Schr\"odinger problem and structural results such as existence, uniqueness, duality and existence of Schr\"odinger-Kantorovich potentials for~\ref{KSPd}. Section \ref{sec: cost res} is devoted to the study of the long-time behaviour of the entropic cost, whereas in Section~\ref{sec: tpike res} we state exponential turnpike estimates for the Fisher information and relative entropy along Schr\"odinger bridges. Section \ref{sec:2} contains preliminary results on the Langevin dynamics and the associated semigroup that are needed for the proof of the main results, that we carry out in Section \ref{sec: proof main res} working at first under an extra regularity assumption on the marginal measures $\mu$ and $\nu$ that we eventually remove thanks to the technical results of Section \ref{sec: approx}.

\subsection{The kinetic Schr\"odinger problem}\label{sec: basic res}

This article is devoted to the analysis of a stochastic mass transportation problem, that we name \emph{kinetic Schr\"odinger problem}, owing to the fact that it is obtained from the classical Schr\"odinger problem by replacing Brownian particles with a system of independent particles following the Langevin dynamics in Schr\"odinger's thought experiment. The first formulation \ref{KSPd}, that we proposed on the basis of Sanov's Theorem, is in terms of an entropy minimisation problem over path probability measures. Besides the change of the reference measure, another difference with respect to classical instances of the Schr\"odinger problem lies in the fact that it is not the full marginal that is constrained at initial and final time, but only its spatial component. Even though \ref{KSPd} seems to be a more faithful representation of Schr\"odinger's thought experiment, also the problem with fully constrained marginals \begin{equation}\label{KFSPd}\tag{KFSPd}
\cC^F_T(\bar\mu,\bar\nu):= \inf\, \biggl\{\cH(\rmP|\rmR): \rmP \in \cP(\Omega), \, (X_{0},V_{0})_{\#}\rmP=\bar\mu,(X_{T},V_{T})_{\#}\rmP=\bar\nu\biggr\}\,.
\end{equation}
where $\bar\mu,\bar\nu\in\cP(\R^{2d})$ is worth studying and we shall work on both problems in the sequel. Through a classical argument \cite{FOLL88} it is possible to reduce the dynamic formulations (cf.~\ref{KSPd} and~\ref{KFSPd}) to static ones. For example, \ref{KSPd} is equivalent to solving
\begin{equation}\label{KSP}\tag{KSP}
 \inf\,\biggl\{\cH(\pi|\rmR_{0,T}): \pi\in\Pi_X(\mu,\,\nu) \biggr\}\,,
\end{equation}
where $\rmR_{0,T}\coloneqq \left((X_0,V_0),\,(X_T,V_T)\right)_{\#}\rmR$ is the joint law of $\rmR$ at initial and terminal time and the set $\Pi_X(\mu,\nu)$ is defined as
\[
\Pi_X\left( \mu,\,\nu\right)\coloneqq\left\{\pi\in\cP\bigl(\R^{2d}\times\R^{2d}\bigr )\mid (\mathrm{proj}_{x_1})_{\#}\pi=\mu,\,(\mathrm{proj}_{x_2})_{\#}\pi=\nu\right\},
\]
with $\mathrm{proj}_{x_i}\bigl((x_1,\,v_1),(x_2,\,v_2)\bigr)\coloneqq x_i$ for any $i=1,2$. In a similar fashion, the static formulation of \ref{KFSPd} is
\begin{equation}\label{KFSP}\tag{KFSP}
 \inf\,\biggl\{\cH(\pi|\rmR_{0,T}): \pi\in\Pi(\bar\mu,\,\bar\nu) \biggr\}\,,
\end{equation}
where $\Pi(\bar\mu,\bar\nu)$ is the (usual) set of couplings of $\bar\mu$ and $\bar\nu$. The equivalence between the static and dynamic formulations is obtained mixing optimal static solutions with the bridges of the reference measure, see \cite{LeoSch} for details. Finally, one can also derive a fluid dynamic (Benamou-Brenier \cite{benamou2000computational}) formulation as well a stochastic control formulation of both problems. In particular, the latter one, that we sketched at \eqref{eq: KSPd stoch control} is the one that motivated us to investigate the turnpike phenomenon. We now proceed to establish some basic though fundamental structural results on the kinetic Schr\"odinger problems at hand. But before doing so, let us present the assumptions under which our main results hold.

\subsubsection{Assumptions}\label{sec: assumptions}
We state here the assumptions on the potential $U$ and on the constraints $\mu,\,\nu,\,\bar\mu$ and $\bar\nu$ that we use in the sequel. We define $\frm_X,\,\frm_V \in \cP(\R^d)$ to be the respectively the space and velocity marginals of $\frm$, in particular $\frm = \frm_X\otimes \frm_V$. 

\begin{enumerate}[label=(H\arabic*),ref=(H\arabic*)]\setlength{\parskip}{5pt}
 \item\label{H1} $U$ is a $C^\oo$ strongly convex potential with bounded derivatives of order $k\geq 2$.

 \item\label{H2} There exist $0<\alpha<\beta$ such that
 \[\sqrt{\beta}-\sqrt{\alpha}\leq \gamma\,,\qquad \text{and}\qquad \alpha\Id_d\leq \nabla^2 U(x)\leq \beta\Id_d\,,\qquad\text{for all }x\in\RD\,,\]
 where $\gamma>0$ is the friction parameter in \eqref{langevin}.
 
 \item\label{H3} The probability measures $\mu$ and $\nu$ on $\R^d$ satisfy \[\cH(\mu|\frm_X)<+\oo\quad\text{and}\quad\cH(\nu|\frm_X)<+\oo\,.\]
 \item\label{H4} $\mu, \nu \,\ll\frm_X$,  $\frac{\De\mu}{\De\frm_X},\,\frac{\De\nu}{\De\frm_X}\in L^\oo(\frm_X)$ and are compactly supported on $\R^{d}$.
 \end{enumerate}
\begin{enumerate}[label=(FH\arabic*),ref=(FH\arabic*)]\setlength{\parskip}{5pt}
\setcounter{enumi}{2} 
 \item\label{FH3}  The probability measures $\bar\mu$ and $\bar\nu$ on $\R^{2d}$ satisfy \[\cH(\bar\mu|\frm)<+\oo\quad\text{and}\quad\cH(\bar\nu|\frm)<+\oo\,.\]
 
 \item\label{FH4} $\bar\mu,\bar\nu \,\ll\frm$, $\frac{\De\bar\mu}{\De\frm},\,\frac{\De\bar\nu}{\De\frm}\in L^\oo(\frm)$ and are compactly supported on $\R^{2d}$.
 \end{enumerate}
 
 Assumption~\ref{H2} implies \emph{local} gradient contraction bounds for the semigroup generated by the Langevin dynamics with a certain rate $\kappa>0$ (see Proposition~\ref{prop:contraction} or \cite{Baudoin}). The exponential rate $\kappa$ of Theorems  \ref{teo entropic turnpike} and \ref{full:teo entropic turnpike} below is precisely the one, computed e.g. in \cite{monmarche2020almost,bolley2010trend}, at which the synchronous coupling is contractive for the (uncontrolled) Langevin dynamics. 

 For each of the main results, we will make it explicit which assumptions from the above list are needed.

\subsubsection{Duality} We begin with a duality result, analogous to the Monge-Kantorovich duality of optimal transport and the more recent dual representations of the entropic cost for the classical Schr\"odinger problem \cite{gigli2018benamou}. 
It is worth noticing that, since the stationary Langevin dynamics is not a reversible measure, $\cC^F_{T}(\cdot,\cdot)$ is not symmetric in its arguments. Nevertheless, due to the ``physical reversibility" of the dynamics \cite{chen2015fast}, that is, reversibility up to a sign flip in the velocities, it is not hard to show that $\cC_T(\cdot,\cdot)$ is symmetric in its arguments.
\begin{prop}\label{lemma:duality}
Grant \ref{H1} and \ref{H3}. Then $\cC_T(\mu,\nu)<\infty$ and 
\begin{equation}\label{dual}
 \cC_T(\mu,\nu)=\sup_{\varphi,\psi\in C_b(\R^d)}\biggl\{\int_{\RD} \varphi\,\De\mu+\int_{\RD}\psi\,\De\nu-\log\int_{\R^{4d}}e^{\varphi\oplus\psi}\,\De \rmR_{0,T}\biggr\}\,.
\end{equation}
Similarly, grant~\ref{H1} and \ref{FH3} it holds   $\cC^F_T(\bar\mu,\bar\nu)<\infty$ and 
\begin{equation}\label{full:dual}
 \cC^F_T\left(\bar\mu,\bar\nu\right)=\sup_{\varphi,\psi\in C_b(\R^{2d})}\biggl\{\int_{\R^{2d}} \varphi\,\De\bar\mu+\int_{\R^{2d}}\psi\,\De\bar\nu-\log\int_{\R^{4d}}e^{\varphi\oplus\psi}\,\De \rmR_{0,T}\biggr\}\,.
\end{equation}
\end{prop}

\subsubsection{The $fg$-decomposition} Optimal couplings in the Schr\"odinger problem are characterised by the fact that their density against the reference measure takes a product form, often called $fg$-decomposition \cite{LeoSch}. In \ref{KSPd} $f$ and $g$ have the additional property of depending only on the first and second space variables respectively.

\begin{prop}\label{fg lemma}
Grant \ref{H1},~\ref{H3}. Then, for all $T>0$, \ref{KSP} and \ref{KSPd} admit unique solutions $\mu^T,\rmP^T$ with $\mu^T= ((X_0,V_0),(X_T,V_T))_{\#}\rmP^T$ and there exist two non-negative measurable functions $f^T,\,g^T$ on $\RD$ such that \begin{equation}\label{fg}\rho^T(x,v,y,w)\coloneqq\frac{\De \mu^T}{\De \rmR_{0,T}}(x,v,y,w)=f^T(x) g^T(y),\qquad \rmR_{0,T}\text{-a.s.}\end{equation}
Moreover, $f^T,g^T$ solve the Schr\"odinger system:
\begin{equation}\label{SSp}
 \begin{cases}
\frac{\De\mu}{\De \frm_X}(x) =f^T(x) \,\bbE_\rmR\big[g^T(X_T)|X_0=x\big],\\
\frac{\De\nu}{\De \frm_X}(y) =g^T(y)\,\bbE_\rmR\big[f^T(X_0)|X_T=y\big].
\end{cases}
\end{equation}
\end{prop}
For \ref{KFSP} and \ref{KFSPd}, the uniqueness of solutions (hereafter $\bar\mu^T$ and $\bar\rmP^T$ respectively) and the $fg$-decomposition are a direct consequence of known results, see e.g. \cite{ruschendorf1993note}, whereas the case \ref{KSP} requires some more work. We remark here that for both dual representation of the cost and the $fg$-decomposition the strict convexity of $U$ and its smoothness are not really necessary, a bounded Hessian would suffice.

\subsection{Long-time behaviour of the entropic costs}\label{sec: cost res} Let us now turn the attention to the ergodic properties of \ref{KSP} and \ref{KFSP} by investigating the long-time behaviour of the entropic cost. To explain the upcoming results, we remark that (H1) implies ergodicity of the Langevin dynamics \cite[Theorem 11.14]{DaPrato} and in particular one has the weak convergence
\begin{equation*}
  R_{0,T} \weakto \frm \otimes \frm .
\end{equation*}
Intuitively, this implies that the variational problem \ref{KSP} converges, in a sense to be made precise, to the   problem
\begin{equation}\label{starimportance}
\min_{\pi\in\Pi_X(\mu,\nu)}\cH(\pi\mid\mathfrak{m}\otimes\mathfrak{m}),
\end{equation}
whose optimal solution  and optimal value are easily seen to be $(\mu\otimes\frm_V)\otimes(\nu\otimes\frm_V)$ and $\cH(\mu|\frm_X)+\cH(\nu|\frm_X)$ respectively. From the point of view of the particle system, this means that in the long-time limit, initial and final states of the system become essentially independent of one another. Moreover, the initial and final velocities are well approximated by independent Gaussians, and are independent from the spatial variables.
The result below turns this intuition into a solid argument, including a quantitative version of the convergence of the entropic cost towards the sum of the marginal entropies. For the classical Schr\"odinger problem, an analogous statement can be found in~\cite{conforti2021formula}.

\begin{theorem}\label{long-time cost}
 Grant~\ref{H1} and~\ref{H3}. Then
\begin{equation}\label{longcost}\lim_{T\to\oo}\cC_T(\mu,\,\nu)=\cH(\mu \mid\frm_X)+\cH( \nu \mid\frm_X)<\oo\,.\end{equation}
Moreover as $T\to\oo$
\begin{equation}\label{weakconv}\mu^T\weakto \left(\mu\otimes\mathfrak{m}_V\right)\otimes\left(\nu\otimes\mathfrak{m}_V\right)\in\Pi_X(\mu,\nu)\,,\end{equation}
weakly and, granted \ref{H2}, there exists a positive constant $C_{d,\alpha,\beta,\gamma}$ (depending only on $d,\alpha,\beta$ and $\gamma$) such that for any $0<\delta\leq1$, as soon as $T>(\tempo)\vee\tempobis$, it holds
	\begin{equation}\label{good:exp conv}
		\abs{\cC_{T}\left(\mu,\nu\right)- \cH\left(\mu|\frm_X\right) -  \cH\left(\nu|\frm_X\right)}\leq C_{d,\alpha,\beta,\gamma}\,\delta^{-3}\,e^{-\kappa\,T}\,\biggl[ \cH\left(\mu|\frm_X\right)+\cH\left(\nu|\frm_X\right)\biggr]\,,\end{equation}and as a consequence the following entropic Talagrand inequality holds
	\begin{equation}\label{gio:talagrand}
	    \cC_{T}\left(\mu,\nu\right) \leq \Big(1+ C_{d,\alpha,\beta,\gamma}\,\delta^{-3}\,e^{-\kappa\,T}\Big)\,\biggl[ \cH\left(\mu|\frm_X\right)+\cH\left(\nu|\frm_X\right)\biggr]\, .
\end{equation}	
\end{theorem}

\begin{remark}
Equation \eqref{weakconv} implies in particular that $\mu^T_0\weakto \mu\otimes\frm_V$ and $\mu^T_T\weakto\nu\otimes\frm_V$. This convergence is also exponential, as we show in Theorem \ref{exp marginal entropies}.
\end{remark}
\begin{theorem}\label{full:teolongcost}
  Under the~\ref{H1} and~\ref{FH3} it holds
 \begin{equation}\label{full:longcost}
  \lim_{T\to\oo}\cC^F_T\left(\bar\mu,\bar\nu\right)= \cH\left(\bar\mu|\frm\right)+\cH\left(\bar\nu|\frm\right)<\oo\,.
 \end{equation}
 Moreover as $T\to\oo$
\begin{equation}\label{full:weakconv}\bar{\mu}^{T}\weakto \bar\mu\otimes\bar\nu\in\Pi\left(\bar\mu,\bar\nu\right)\,,\end{equation}
weakly and, granted \ref{H2}, there exists a positive constant $C_{d,\alpha,\beta,\gamma}$ such that for any $0<\delta\leq1$, as soon as $T>(\tempo)\vee\tempobis$, it holds
	\begin{equation}\label{good:full:exp conv}
	    \abs{\cC^F_{T}\left(\bar\mu,\bar\nu\right)- \cH\left(\bar\mu|\frm\right) -  \cH\left(\bar\nu|\frm\right)} \leq C_{d,\alpha,\beta,\gamma}\,\delta^{-3}\,e^{-\kappa\,T}\,\biggl[ \cH\left(\bar\mu|\frm\right)+\cH\left(\bar\nu|\frm\right)\biggr]\, ,
\end{equation}	and as a consequence the following entropic Talagrand inequality holds
	\begin{equation}\label{gio:full:talagrand}
	    \cC^F_{T}\left(\bar\mu,\bar\nu\right) \leq \Big(1+ C_{d,\alpha,\beta,\gamma}\,\delta^{-3}\,e^{-\kappa\,T}\Big)\,\biggl[ \cH\left(\bar\mu|\frm\right)+\cH\left(\bar\nu|\frm\right)\biggr]\, .
\end{equation}
\end{theorem}

The proof of the qualitative statements in the above results rely on $\Gamma$-convergence and some simple consequences of the heat kernel estimates in \cite{DelarueMenozzi}.  The key ingredient in the proof of the exponential estimates is a representation formula for the difference 
\[\cC_{T}\left(\mu,\nu\right)- \cH\left(\mu|\frm_X\right) -  \cH\left(\nu|\frm_X\right) \] 
 that we establish at Lemma \ref{lemmarapprcosto} and allows to profit from the turnpike estimates at Theorem~\ref{teo entropic turnpike} and \ref{full:teo entropic turnpike} below. 
 
\subsection{Long-time behaviour of Schr\"odinger bridges}\label{sec: tpike res}

One of the main contributions of this article are the upcoming quantitative results on the long-time behaviour of Schr\"odinger bridges, which imply in particular exponential convergence to $\frm$ when looking at timescales of order $T$ and exponential convergence in $T$ to the Langevin dynamics  when looking at the Schr\"odinger bridge over a fixed time-window $[0,t]$.

\subsubsection{Entropic turnpike property}

We propose two turnpike results in which distance from equilibrium is measured through the relative entropy $\cH(\cdot|\frm)$ and the Fisher information $\cI(\cdot)$, see \eqref{def:Fisher info} below. The use of $\cH(\cdot|\frm)$ is natural in light of the fact that the costs $\cC_T(\mu,\nu)$ and $\cC^F_T(\bar\mu,\bar\nu)$ are also relative entropies, but computed on different spaces. On the other hand, the bound on $\cI(\cdot)$ is reminiscent of the celebrated Bakry-\'Emery estimates \cite{BAKEM}. It is worth noticing that entropic turnpike estimates seem to be very rare in the existing literature and even less so are bounds on the Fisher information: we shall elaborate more on this at Remark \ref{rem:Fisher}. The key assumption for obtaining \eqref{entropic eq} and \eqref{full:entropic eq} is~\ref{H2}, asking $U$ to be strongly convex and such that the difference between the smallest and largest eigenvalues of $\nabla^2U(x)$ is controlled by the friction parameter $\gamma$ uniformly in $x$. This assumption is often encountered in works dealing with the long-time behaviour of the (uncontrolled) kinetic Fokker-Planck equation, see e.g. \cite{bolley2010trend}. Although exponential $L^2$ estimates are known to hold under considerably weaker assumptions (see e.g. \cite{herau2004isotropic} and \cite{contract:singular:pot,erg:lyap:langevin} for singular potentials), and entropic estimates assuming a bounded and positive Hessian have been known for more than a decade \cite{cedric2009hypocoercivity}, it is only recently \cite{guillin2021kinetic} that entropic estimates have been obtained beyond the bounded Hessian case. In light of this, the question of how to improve our results is quite interesting and deserves to be further investigated. Let us state the announced results, beginning with \ref{KSPd}. To do so, we need another bit of notation: if $\rmP^T$ is the unique solution of \ref{KSPd}, we call \emph{entropic interpolation} $(\mu^T_{t})_{t\in[0,T]}$ the marginal flow of $\rmP^T$ and denote $\rho^T_t$ its density against $\frm$, i.e.
\begin{equation*}
\forall t\in[0,T],\quad    \mu^T_{t} = (X_t,V_t)_{\#}\rmP^T,\qquad \rho^T_t\coloneqq\frac{\De\mu^T_t}{\De\frm}\,.
\end{equation*}
With the obvious small modifications, we also define the entropic interpolation $({\bar\mu}^T_{t})_{t\in[0,T]}$ and their densities $({\bar\rho}^T_{t})_{t\in[0,T]}$ in the framework of \ref{KFSPd}.
Furthermore, we introduce the functional $\cI$ to be the Fisher information with respect to $\frm$, defined for any $q\ll\frm\in\cP(\R^{2d})$ as
\begin{equation}\label{def:Fisher info} \cI(q) := \begin{cases} \int_{\bbR^{2d}} \abs{\nabla \log \frac{\De q}{\De \frm}}^2\, \De q & \quad \mbox{if $\nabla \log \frac{\De q}{\De \frm}\in L^2(q)$,}  \\ +\infty, & \quad \mbox{otherwise.} \end{cases}.\end{equation}

\begin{theorem}[Entropic turnpike for KSP]\label{teo entropic turnpike}
Grant \ref{H1}, \ref{H2} and \ref{H3}. There exists a positive constant $C_{d,\alpha,\beta,\gamma}$ such that for any $0<\delta\leq1$ and $t\in[\delta,\,T-\delta]$,
as soon as $T>\tempo$, it holds
\begin{equation}\label{fisher eq}
  \cI(\mu^T_{t})\leq C_{d,\alpha,\beta,\gamma}\,  \delta^{-3}\, e^{-2\kappa[t\wedge(T-t)]}\,\cC_T(\mu,\nu)\,,
 \end{equation}
 \begin{equation}\label{entropic eq}
  \cH(\mu^T_{t}|\frm)\leq C_{d,\alpha,\beta,\gamma}\,  \delta^{-3}\, e^{-2\kappa[t\wedge(T-t)]}\,\cC_T(\mu,\nu)\,.
 \end{equation}
Moreover, as soon as $T>(\tempo)\vee\tempobis$, we have
 \begin{equation}\label{bis:entropic eq}
  \cH(\mu^T_{t}|\frm)\leq C_{d,\alpha,\beta,\gamma}\,  \delta^{-3}\, e^{-2\kappa[t\wedge(T-t)]}\,\bigg[\cH(\mu|\frm_X)+\cH(\nu|\frm_X)\bigg]\,.
 \end{equation}
\end{theorem}
\begin{theorem}[Entropic turnpike for KFSP]\label{full:teo entropic turnpike}
Grant \ref{H1}, \ref{H2} and \ref{FH3}.There exists a positive constant $C_{d,\alpha,\beta,\gamma}$ such that for any $0<\delta\leq1$ and $t\in[\delta,\,T-\delta]$, as soon as $T>\tempo$, it holds
\begin{equation}\label{full:fisher eq}
  \cI(\bar\mu^T_{t})\leq C_{d,\alpha,\beta,\gamma}\,  \delta^{-3}\, e^{-2\kappa[t\wedge(T-t)]}\,\cC_T^F(\bar\mu,\bar\nu)\,,
 \end{equation}
 \begin{equation}\label{full:entropic eq}
  \cH(\bar{\mu}^T_{t}|\frm)\leq C_{d,\alpha,\beta,\gamma}\,  \delta^{-3}\, e^{-2\kappa[t\wedge(T-t)]}\,\cC_T^F(\bar\mu,\bar\nu)\,.
 \end{equation}
 Moreover, as soon as $T>(\tempo)\vee\tempobis$, we have
  \begin{equation}\label{bis:full:entropic eq}
  \cH(\bar{\mu}^T_{t}|\frm)\leq C_{d,\alpha,\beta,\gamma}\,  \delta^{-3}\, e^{-2\kappa[t\wedge(T-t)]}\,\bigg[\cH(\bar\mu|\frm)+\cH(\bar\nu|\frm)\bigg]\,.
 \end{equation}
\end{theorem}

\begin{remark}\label{rem: tunrpike literature}
If we compare our results with what is known in deterministic control we remark that, quite curiously, exponential estimates for the deterministic noiseless version of \eqref{eq: KSPd stoch control}, obtained removing the Brownian motion form the controlled state equation, do not seem to be covered from existing results, even in the case when $\mu$ and $\nu$ are Dirac measures.\footnote{For example, if we compare with the reference work \cite{trelat2015turnpike}, the matrix $W$ defined at Eq.~(10) therein would not be invertible for the problem under consideration, which thus fails to satisfy the hypothesis of the main turnpike result obtained there.} For linear-quadratic problems though,  the result is well known, see e.g. \cite{breiten2020turnpike} for precise estimates. Theorems \ref{teo entropic turnpike} and \ref{full:teo entropic turnpike} provide \emph{global} turnpike estimates, that is to say we do not ask $\mu$ and $\nu$ to be close to $\frm$. We do ask $\cH(\mu|\frm),\cH(\nu|\frm)<+\infty$, but this condition is very mild and necessary for the Schr\"odinger problem to have a finite value. This is in contrast with most exponential turnpike estimates we are aware of in deterministic control (see e.g. \cite[Theorem 1]{trelat2015turnpike}). The passage from local to global estimates seems to be possible \cite{trelat2018integral,trelat2020linear} under some extra assumptions, such as the existence of a storage function, but this comes at the price of losing  quite some information on the multiplicative constants appearing in \eqref{entropic eq}. Moreover, the condition $T>\tempo$ of Theorem~\ref{teo entropic turnpike} should be replaced with a condition of the form $T>T_0$ with $T_0$ depending on the initial conditions and potentially very large. 
\end{remark}

\begin{remark}\label{rem:Fisher}  The bound on the Fisher information is our strongest result as it implies immediately an entropic bound thanks to the logarithmic Sobolev inequality \eqref{LS}. Moreover, entropic bounds are stronger than bounds expressed by means of a transport distance such as $W_1$ or $W_2$, since $\frm$ satisfies Talagrand's inequality \eqref{talagrand}.
\end{remark}

\begin{remark}
Proving the turnpike property for Schr\"odinger bridges in this context is harder than in the classical setting, and we need to work under stronger assumptions on the potential $U$ than its strong convexity. This is not a surprise. Indeed, proving the exponential convergence to equilibrium for the kinetic Fokker-Planck equation is a difficult problem that  has been, and still is, intensively studied by means of either a probabilistic or an analytic approach, see \cite{cattiaux2019entropic,eberle2019couplings,talay2002stochastic,guillin2021kinetic} for some references on the probabilistic approach.
Following the terminology introduced by Villani in his monograph \cite{cedric2009hypocoercivity}, this obstruction is a manifestation of the \emph{hypocoercive} nature of the kinetic Fokker-Planck equation. \ref{KSP}  may indeed be regarded as the prototype of an hypocoercive stochastic control problem.  For the moment, we have been able to show the turnpike property under a quasilinearity assumption.
Assumptions of this type, where the friction parameter has to be in some sense large in comparison with the spectrum of $\nabla^2U$ are commonly encountered in the literature. In the language of probability, they ensure that the synchronous coupling is contracting for the Langevin dynamics \cite{bolley2010trend,monmarche2020almost}. On the other hand, from an analytical standpoint, Assumption~\ref{H2} implies \emph{local} gradient bounds for the semigroup generated by the Langevin dynamics \cite{Baudoin}. Finally, we recall that the exponential rate $\kappa$ of Theorems \ref{teo entropic turnpike} and \ref{full:teo entropic turnpike} is precisely the one, computed e.g. in \cite{monmarche2020almost,bolley2010trend}, at which synchronous coupling is contractive for the (uncontrolled) Langevin dynamics. 
\end{remark}

\paragraph{\bf Proof strategy} 
A general idea to obtain exponential speed of convergence to equilibrium for hypocoercive equations systematically exploited in \cite{cedric2009hypocoercivity} is that of modifying the "natural" Lyapunov function of the system by adding some extra terms in such a way that proving exponential dissipation becomes an easier task. For the Langevin dynamics, a suitable modification of the natural Lyapunov functional, that is the relative entropy $\cH(\cdot|\frm)$, is obtained considering
\begin{equation*}
\mu\mapsto a     \cH(\mu|\frm) +\cI(\mu)
\end{equation*}
for a carefully chosen constant $a>0$. Emulating Bakry-\'Emery $\Gamma$-calculus \cite{Baudoin} it is possible to show that the modified Lyapunov functional decays exponentially along solutions of the kinetic Fokker-Planck equation. Our proof of the turnpike property consists in implementing this abstract idea on the $fg$-decomposition of the entropic interpolation, as we now briefly explain. Indeed, in order to bound $\cI(\mu^T_t)$ one is naturally led to consider the quantities 
\begin{subequations}
\begin{equation}\label{bward corr}
\int_{\R^{2d}}\abs{\nabla \log f_s^T}^2 f_s^T g_s^T \,\De \frm\,,
\end{equation}
\begin{equation}\label{fward corr}
\int_{\R^{2d}} \abs{\nabla \log g_s^T}^2f_s^T g_s^T\,\De \frm\,.
\end{equation}
\end{subequations}
However, it is not clear how to obtain a differential inequality ensuring exponential (forward) dissipation of \eqref{bward corr} and exponential (backward) dissipation of \eqref{fward corr}. But, as we show at Lemma \ref{eq:correctors}, it is possible to find two norms $|\cdot|_{M^{-1}}$ and $|\cdot|_{N^{-1}}$, that are  equivalent to the Euclidean norm and such that if we define
\begin{equation}\label{defcorrectors intro}
	\varphi^T(s) \coloneqq \int_{\R^{2d}}\abs{\nabla \log f_s^T}^2_{N^{-1}} f_s^Tg_s^T \,\De \frm \quad\mbox{and}\quad \psi^T(s) \coloneqq \int_{\R^{2d}} \abs{\nabla \log g_s^T}^2_{M^{-1}} f_s^Tg_s^T \,\De \frm\,,
\end{equation}
then $\varphi^T(s)$ and $\psi^T(s)$ satisfy the desired exponential estimates. To complete the proof, one needs to take care of the boundary conditions. This part is non trivial as it demands to prove certain regularity properties of the $fg$-decomposition and it is accomplished in two steps: we first show at Proposition~\ref{prop: regularizing effect} a regularising property of entropic interpolations, namely that if $\cH(\mu|\frm_X),\cH(\nu|\frm_X)$ are finite, then the Fisher information $\cI(\mu^T_t)$ is finite for any $t\in(0,T)$. The proof of this property is based on a gradient bound obtained in \cite{GuillinWang} and is of independent interest. The second step (Proposition \ref{gio:prop:T/2}) consists in showing that for a fixed small $\delta$, $\varphi^T(\delta)$  and $\psi^T(T-\delta)$ can be controlled with by the sum of $\cI(\mu^T_{\delta})$ and $\cI(\mu^T_{T-\delta})$. We prove this estimate adapting an argument used in \cite{trelat2015turnpike} in the analysis of deterministic finite dimensional control problems.

\medskip
\subsubsection{Convergence to the Langevin dynamics over a fixed time-window}

We are able to precisely analyse the behaviour of entropic interpolations for a fixed time $t$, while $T$ grows large. More precisely, we show that the (uncontrolled) Langevin dynamics and the Schr\"odinger bridge are exponentially close in the long-time regime $T\to\infty$, for all time-windows $[0,t]$. Note that this result cannot be deduced from the turnpike estimates of the former section. 

\begin{theorem}\label{short time}
Under hypotheses~\ref{H1},~\ref{H2} and \ref{H3}, there exists a positive constant $C_{d,\alpha,\beta,\gamma}$ such that for any $0<\delta\leq 1$ and  $t\in[0,\,T-\delta]$, as soon as $T>\tempo$, it holds
\[\cW_{2}(\mu^T_{t},\mu^\oo_{t})\leq C_{d,\alpha,\beta,\gamma}\,\delta^{-\frac{3}{2}}\,e^{-\kappa(T-t)}\,\sqrt{\cC_T(\mu,\nu)}\,,\]
$\mu^\oo_t$ is the law of $(X_t,\,V_t)$ satisfying
\begin{equation}\label{eq: short time langevin}
    \begin{cases}
     \De X_t = V_t\De t,\\
     \De V_t = -\nabla U(X_t)\De t -\gamma V_t\De t + \sqrt{2\gamma}\,\De B_t,\\
     (X_0,V_0)\sim \mu\otimes\frm_V.
    \end{cases}
\end{equation}
\end{theorem}

A similar statement holds true for \ref{KFSP} replacing Assumption \ref{H3} with \ref{FH3} and with initial condition in \eqref{eq: short time langevin} given by $\bar\mu$.

\section{Preliminaries}\label{sec:2}

In this section we collect useful results about the Markov semigroup associated to the kinetic Fokker-Planck equation.
In what follows we write $\lesssim$ to indicate  that an inequality holds up to a multiplicative positive constant depending possibly on the dimension $d$, the bounds on the spectrum of $\nabla^2 U$, $\alpha$ and $\beta$, or the friction parameter $\gamma$.  

\subsection{On the assumptions} In this short section we report some straightforward consequences of the various assumptions listed at Section \ref{sec: assumptions} that we shall repeatedly use from now on. We begin by observing that assumption \ref{H1} guarantees that $\frm\in\cP_2(\R^{2d})$ and that $\frm_X$ satisfies Talagrand's inequality because of \cite[Corollary 9.3.2]{bakry2013analysis}, i.e. for any $q\in\cP(\R^{d})$
 \begin{equation}\label{talagrandx}\cW_2(q,\,\frm_X)^2\lesssim\cH(q|\frm_X)\,.\end{equation}
 Since the Talagrand inequality holds also for the Gaussian measure $\frm_V$, from \cite[Proposition 9.2.4]{bakry2013analysis}  it follows that for any $q\in\cP(\R^{2d})$
  \begin{equation}\label{talagrand}\cW_2(q,\,\frm)^2\lesssim\cH(q|\frm)\,.
  \end{equation}
 Let us also point out that~\ref{H4} implies~\ref{H3} and that under \ref{H1} and \ref{H3} it easily follows that $\mu,\,\nu\in\cP_2(\R^d)$. Indeed,
 \begin{equation}\label{secondmomentbound}\int_{\RD} \abs{x}^2\De\mu\lesssim \int_{\RD}\abs{x}^2\De\frm_X+\cW_2(\mu,\frm_X)^2\overset{\eqref{talagrandx}}{\lesssim}\int_{\RD}\abs{x}^2\De\frm_X+\cH(\mu|\frm_X)<+\oo\,,\end{equation}
 and similarly for the measure $\nu$. We also remark that \ref{FH4} implies \ref{FH3}.  Moreover, from \ref{H1} and \ref{FH3}, by means of \eqref{talagrand}, it follows that $\bar\mu,\,\bar\nu\in\cP_2(\R^{2d})$.

Finally, let us also notice that \ref{H1} and \ref{H2} guarantee  the validity of a log-Sobolev inequality for $\frm_X$ because of \cite[Corollary 5.7.2]{bakry2013analysis}, and by means of \cite[Proposition 5.2.7 and Proposition 5.5.1]{bakry2013analysis} it follows that $\frm$ satisfies a log-Sobolev inequality. Therefore for any $q\ll\frm$ it holds
\begin{equation}\label{LS}
\cH(q|\frm)\lesssim \cI(q)\,.
\end{equation}

\subsection{Markov semigroups and heat kernel} The generator $L$ associated to the SDE~\eqref{langevin} is given by
\begin{equation*}
  L=\gamma\Delta_v-\gamma v\cdot\nabla_v -\nabla U\cdot\nabla_v+v\cdot\nabla_x
\end{equation*}
 while its adjoint in $L^2(\frm)$ reads as
 \begin{equation*}
  L^*=\gamma\Delta_v-\gamma v\cdot\nabla_v +\nabla U\cdot\nabla_v -v\cdot\nabla_x\,.
 \end{equation*}
 Under assumption~\ref{H1}, it is well known that H\"ormander's Theorem for parabolic hypoellipticity applies~\cite[Theorem 1.1]{hormander1967hypoelliptic} to the operator $L$ , and thus the associated semigroup $(P_t)_{t\geq0}$ admits a probability kernel $p_t((x,y),(y,w))$, which is $C^\infty$ in all of the parameters, with respect to the invariant probability measure \begin{equation*}
     \De \frm (x,v)=  \frac{1}{Z}  e^{-U(x) - \frac{|v|^2}{2}}\,\De x \De v,
 \end{equation*}
 where $Z$ is a normalising constant. Sometimes, with a slight abuse of notation we will write $\frm(x,v)$ to denote the density of $\frm$ with respect to the Lebesgue measure. Similarly, we will denote by $(P_t^*)_{t\geq0}$ the semigroup associated to $L^*$. Note that the function $p_t$ also represents the density of $\rmR_{0,t}$ (the joint law at time $0$ and $t$ of the solution to~\eqref{langevin}) with respect to $\De \frm \otimes \frm$. Moreover, according to~\cite[Theorem 1.1]{DelarueMenozzi}, $p_t(\cdot,\cdot)$ satisfies two-sided Gaussian estimates.
 Importantly, $p_t$ is locally bounded away from zero and infinity, but with constants that might depend non-trivially on the time horizon $T$.

For some of our proofs, we need lower bounds that are uniform in $T$. To this aim, we have the following consequence of the results of \cite{DelarueMenozzi}, whose proof is postponed to the appendix. 
\begin{lemma}\label{lemmaappendicebound}
  Let $T_0>0$ be fixed. Under assumption~\ref{H1}, there exists a constant $c_{T_0}>0$ such that for all $T\geq T_0$ and all $(x,v),\,(y,w)\in\R^{2d}$
  \begin{equation}\label{boundlogappendice}\log p_T\left((x,v),\,(y,w)\right)\geq -c_{T_0}\Bigl(1+\abs{x}^2+\abs{v}^2+\abs{y}^2+\abs{w}^2\Bigr)\,.\end{equation}
\end{lemma}
 We stress that the Langevin dynamics~\eqref{langevin} are not reversible, and in particular the probability kernel $p_t$ is not symmetric. However, it is symmetric up to a sign-flip in the velocities,
 \begin{equation}\label{physrever}
 p_t \bigl((x,v),(y,w)\bigr)=p_t\bigl((y,-w),(x,-v)\bigr)\quad\forall t \geq 0,\quad\forall (x,v),\,(y,w)\in\R^{2d}\,.
 \end{equation}
 As we said above, this useful property is sometimes called physical reversibility.
 
\subsection{Contraction of the semigroup}
In our setup, due to the lack of a curvature condition, the standard Bakry-Emery machinery does not apply to obtain a commutation estimate for the semigroup of the type
\begin{equation}\label{eq:commestimates}
     \abs{\nabla P_t h(z)}\leq e^{-c \,t}\,P_t\big(\abs{\nabla h}\big)(z)\,,
\end{equation}
for some $c>0$.
It is still possible to obtain a commutation estimate similar to~\eqref{eq:commestimates} by replacing the Euclidean norm $|\cdot|$ by a certain twisted norm $\abs{\xi}_M\coloneqq\sqrt{\xi\cdot M\xi}$ on $\R^{2d}$ for some  well chosen positive definite symmetric matrix $M\in\R^{2d\times 2d}$. This is a common idea in the kinetic setting and it is exploited for example in~\cite{Baudoin, guillin2021kinetic, monmarche2020almost}. 

For instance, in Theorem 1 of~\cite{monmarche2020almost} the author studies the contraction properties of the semigroup $P_t$ associated to the SDE on $\R^m$
\begin{equation}\label{genericSDE}
 \De Z_t=b(Z_t)\De t +\Sigma\De B_t\,,
\end{equation}
with the drift $b:\R^m\to\R^m$ being globally Lipschitz and $\Sigma$ a constant positive-semidefinite symmetric matrix. The author shows that the condition on the Jacobian matrix $J_b$ of the drift 
\begin{equation}\label{eq:contR}
    \xi\cdot(MJ_b(z))\xi\leq -\kappa\, \xi\cdot M\xi=-\kappa\,\abs{\xi}_M^2\qquad \forall\xi \in \R^m,\, \forall z\in\R^m\,,
\end{equation}
where $\kappa\in \R$ and $M$ is a positive definite symmetric matrix, is equivalent to the commutation estimate
\begin{equation*}
    \abs{\nabla P_t h(z)}_{M^{-1}}\leq e^{-\kappa\,t}\,P_t\big(\abs{\nabla h}_{M^{-1}}\big)(z)\,.
\end{equation*}
Our setup, which is also discussed in~\cite[Section 3.3]{monmarche2020almost},  corresponds to the choice $m = 2d$, and
\begin{equation*}
 b(x,v)=\begin{pmatrix}
         v\\
         -\nabla U(x)-\gamma v
        \end{pmatrix}\,\qquad\Sigma=\begin{pmatrix}
         0 &0\\
         0& \sqrt{2\gamma}\Id
        \end{pmatrix}
\end{equation*}
and therefore the Jacobian reads as
\begin{equation*}
 J_b(x,v)=\begin{pmatrix}
  0  &\Id\\
  -\nabla^2 U(x) &-\gamma\Id
 \end{pmatrix}\,.
\end{equation*}
In~\cite[Proposition 5]{monmarche2020almost}, the author shows that~\eqref{eq:contR} holds with $\kappa>0$ as long as $\alpha$ and $\beta$ from assumption~\ref{H2} are close enough. By exploiting the symmetry of the heat kernel up to a sign flip, we obtain a similar commutation estimate also for the reversed dynamics.

In view of the above discussion we have the following. 
\begin{prop}\label{prop:contraction} Assume that~\ref{H1} and~\ref{H2} hold. Then, there exist a constant $\kappa >0$ and positive definite symmetric matrices $M,N \in \R^{2d\times 2d}$  such that~\eqref{eq:contR} holds and
\begin{enumerate}[label=(\roman*)]
    \item  For all $h\in C^1_c(\R^{2d})$, $t\geq0$ and $z\in\R^{2d}$
\begin{equation}\label{contrsem}
 \abs{\nabla P_t h(z)}_{M^{-1}}\leq e^{-\kappa\,t}\,P_t\big(\abs{\nabla h}_{M^{-1}}\big)(z)\,.
\end{equation}
    \item  For all $h\in C^1_c(\R^{2d})$, $t\geq0$ and $z\in\R^{2d}$
\begin{equation}\label{contrsemadj}
 \abs{\nabla P^\ast_t h(z)}_{N^{-1}}\leq e^{-\kappa\,t}\,P^\ast_t\big(\abs{\nabla h}_{N^{-1}}\big)(z)\,.
\end{equation}
\end{enumerate}
\end{prop}
\begin{proof} A proof of~\eqref{eq:contR}  with $\kappa > 0$ under~\ref{H1} and~\ref{H2} can be obtained by mimicking the computations in Theorem 2.12 of~\cite{Baudoin}, where the case $\gamma=1$ is discussed. Given~\eqref{eq:contR}, (i) follows from Theorem~1 in~\cite{monmarche2020almost}.

We now derive (ii) from (i) with the help of~\eqref{physrever}. For any function $f$ on $\R^{2d}$ define the transformation $\cS f(x,v) = f(x,-v)$ and set 
\begin{equation*}
N= \begin{pmatrix}
                                \Id &0\\
                                0 &-\Id                                                                        \end{pmatrix} M \begin{pmatrix}
                                \Id &0\\
                                0 &-\Id                                                                        \end{pmatrix}\,.
\end{equation*}
Note that $\cS^2 = \Id$, moreover in view of~\eqref{physrever}, for all $h\in C_c^1(\R^{2d})$, $P_t^\ast (\cS h) = \cS(P_t h)$ and $\cS|\nabla h|_{M^{-1}} = |\nabla(\cS h)|_{N^{-1}}$. It is then immediate to derive
\begin{equation*}
    \abs{\nabla P^\ast_t h}_{N^{-1}} = \cS \abs{\nabla P_t (\cS h)}_{M^{-1}} \leq e^{-\kappa\,t}\,\cS \Big(P_t\big(\abs{\nabla (\cS h)}_{M^{-1}}\big)\Big) = e^{-\kappa\,t}\,P^\ast_t\big(\abs{\nabla h}_{N^{-1}}\big)\,,
\end{equation*}
which is the desired conclusion.
\end{proof}

As a result of Proposition~\ref{prop:contraction} and Theorem 1 in~\cite{monmarche2020almost} we have the equivalent statements, with $M,N$ and $\kappa>0$ as above, and all $q_1,q_2\in \cP(\bbR^{2d})$, 
\begin{equation}\label{BEcontW}
\cW_{M,2}(q_1 P_t, q_2 P_t)\leq e^{-\kappa\,t}\,\cW_{M,2}(q_1, q_2)\,,
\end{equation}
\begin{equation}\label{BEcontstarW}
\cW_{N,2}(q_1 P^\ast_t, q_2 P^\ast_t)\leq e^{-\kappa\,t} \cW_{N,2}(q_1, q_2)\,,
\end{equation}
where $\cW_{M,2}(q_1,q_2)$ is the $\cW_2$-Wasserstein distance on $\cP(\R^{2d})$ with the Euclidean metric replaced by $d_M(x,y) = |x-y|_M$ and similarly for $\cW_{N,2}(q_1,q_2)$.
\section{Proof of the main results}\label{sec: proof main res}

\subsection{Duality and \texorpdfstring{$fg$}{fg}-decomposition for KSP}\label{fg:KSP} 
\begin{proof}[Proof of Proposition~\ref{fg lemma}] We only sketch the proof as it is rather standard. We consider the measure  $\rmR_{0,T}^X\coloneqq (\mathrm{proj}_{x_1},\mathrm{proj}_{x_2})_{\#}\rmR_{0,T}=(X_0,X_T)_{\#}\rmR$ and the minimisation problem,
\begin{equation}\label{Xproblem}
 \min_{q\in\Pi(\mu.\nu)} \cH\left(q|\rmR_{0,T}^X\right)\,,
\end{equation}
where $\Pi(\mu.\nu)$ is the set of couplings of $\mu, \nu \in \cP(\bbR^d\times\bbR^d)$.
In view of the heat kernel lower bound in Lemma~\ref{lemmaappendicebound}, we know that for some $C>0$ and uniformly in $x,y$ it holds
\begin{equation*}
  \frac{\De \rmR_{0,T}^X}{\De(\frm_X\otimes\frm_X)}(x,y)\geq \frac{1}{C} e^{-C\left(1+\abs{x}^2+\abs{y}^2\right)}\,.
\end{equation*}
which in combination with~\ref{H3} implies that $\cH(\mu\otimes\nu|\rmR_{0,T}^X)<\oo$. Indeed the bound above implies that for any $T>T_0$, $T_0$ fixed, there is a constant $C_{d,\alpha,\beta,\gamma,T_0}>0$ such that
 \begin{equation}\label{bound scemo 2}
  \begin{aligned}
  \cH(\mu\otimes\nu|\rmR_{0,T}^X) =&\,\cH(\mu\otimes\nu|\frm_X\otimes\frm_X)-\int_{\R^{4d}}\log  \frac{\De \rmR_{0,T}^X}{\De(\frm_X\otimes\frm_X)}\,\De\mu\otimes\nu\\
   \overset{\eqref{secondmomentbound}}{\leq}&\,C_{d,\alpha,\beta,\gamma,T_0}\big[1+\cH(\mu|\frm_X)+\cH(\nu|\frm_X)\big]\,.
  \end{aligned}
 \end{equation}
Thus, Proposition 2.5 in~\cite{LeoSch} applies and the above minimisation problem has indeed a unique solution $\pi \in \Pi(\mu,\,\nu)$. By applying~\cite[Proposition 2.1]{GigTam21}, there exist two non-negative measurable functions $f^T,\,g^T$ on $\RD$ such that
\begin{equation}\label{eq:radon pi}
    \frac{\De \pi}{\De \rmR_{0,T}^X} (x,\,y) = f^T(x) g^T(y),\qquad \rmR_{0,T}^X\text{-a.s.,}
\end{equation}
from which~\eqref{SSp} directly follows.
Now, in view of the additive property of the relative entropy, we get for any $\rmP\in \cP(\Omega)$
\begin{equation*}
 \cH(\rmP|\rmR)=\cH\left(\rmP_{0,T}^X|\rmR_{0,T}^X\right)+\int_{\R^{2d}} \cH\left(\rmP^{x,y}|\rmR^{x,y}\right)\De\rmP^X_{0,T}(x,y),
\end{equation*} with $\rmR^{x,y} = \rmR(\,\cdot\,|X_0 = x, X_T = y)$ and similarly for $\rmP^{x,y}$.
Therefore, a minimizer to~\ref{KSPd} can be found by defining
\begin{equation*}
    \rmP^T(\cdot) = \int_{\R^{d}\times \R^{d}} \rmR(\,\cdot\,|X_0 = x, X_T = y)\, \De \pi(x,\,y),
\end{equation*}
which satisfies $(X_0,X_T)_{\#}\rmP^T = \pi$ and $\cC_T(\mu,\nu) = \cH(\rmP^T|\rmR) = \cH(\pi|\rmR^X_{0,T})<\infty$. In particular, in view of~\eqref{bound scemo 2}, for all $T>T_0$, $T_0$ fixed, there is $C_{d,\alpha,\beta,\gamma,T_0}>0$ such that 
\begin{equation}\label{boundscemocosto}
      \cC_T(\mu,\nu) \leq C_{d,\alpha,\beta,\gamma,T_0}\big[1+\cH(\mu|\frm_X)+\cH(\nu|\frm_X)\big]\,.
\end{equation}
Similarly, for any $q\in \Pi_X(\mu,\nu)$, denoting $q^X = (\mathrm{proj}_{x_1},\mathrm{proj}_{x_2})_\#q$, we have $\cH(q|\rmR_{0,T})\geq \cH(q^X|\rmR^X_{0,T}) \geq \cH(\pi|\rmR^X_{0,T})$ with equality if and only if $q = \mu^T$ where
\begin{equation}\label{Xrel}
    \mu^T(\cdot) = \int_{\R^d\times \R^d}\rmR_{0,T}(\,\cdot\,|X_0 = x, X_T = y)\, \De \pi(x,\,y).
\end{equation}
By construction $\mu^T= ((X_0,V_0),(X_T,V_T))_{\#}\rmP^T$ and $\cH(\rmP^T|\rmR)=\cH(\mu^{T}|\rmR_{0,T})= \cH(\pi|\rmR^X_{0,T})<\infty$.
The solutions are unique by strict convexity of the entropy and the linearity of the constraint. Equation~\eqref{Xrel} implies equality of the conditional distributions of $\mu^T$ and $R_{0,T}$ given the space variables. But then, 
\begin{equation*}
  \frac{\De \mu^{T}}{\De \rmR_{0,T}}(x,v,y,w) = \frac{\De (\mathrm{proj}_{x_1},\mathrm{proj}_{x_2})_{\#}\mu^{T}}{\De (\mathrm{proj}_{x_1},\mathrm{proj}_{x_2})_{\#}\rmR_{0,T}}(x,y) =  \frac{\De \pi}{\De \rmR^X_{0,T}}(x,y)=f^T(x)g^T(y),\qquad \rmR_{0,T}\text{-a.s.}
\end{equation*}
\end{proof}

\begin{proof}[Proof of Proposition \ref{lemma:duality}]
We have already seen in the previous proof that $\cC_T(\mu,\nu)$ is finite. Now, since \ref{KSP} is equivalent to the minimisation problem~\eqref{Xproblem}, from \cite[Proposition 6.1]{LeoMinEner} it follows
\begin{equation*}\begin{aligned}
    \cC_T(\mu,\nu)=&\,\sup_{\varphi,\psi\in C_b(\R^d)}\biggl\{\int_{\RD}\varphi\,\De\mu+\int_{\RD}\psi\,\De\nu-\int_{\R^{2d}} \left(e^{\varphi\oplus\psi}-1\right)\De\rmR_{0,T}^X\biggr\}\\\leq&\sup_{\varphi,\psi\in C_b(\R^d)}\biggl\{\int_{\RD}\bigl(\varphi\oplus\psi\bigl)\,\De\pi-\log\int_{\R^{2d}}e^{\varphi\oplus\psi}\,\De\rmR_{0,T}^X\biggr\}\\
    \leq&\, \sup_{h\in C_b(\R^{2d})}\biggl\{\int_{\R^{2d}}h\, \De\pi-\log\int_{\R^{2d}}e^{h}\,\De\rmR_{0,T}^X\biggr\}\overset{(\dagger)}{=}\cH\bigl(\pi|\rmR_{0,T}^X\bigr)=\cC_T(\mu,\nu)\,,
\end{aligned}\end{equation*}
where $\pi$ is the unique optimizer in \eqref{Xproblem}, while $(\dagger)$ is the Donsker-Varadhan variational formula \cite[Lemma 1.4.3a]{DupuisEllis}. This concludes the proof since $\rmR_{0,T}^X\coloneqq (\mathrm{proj}_{x_1},\mathrm{proj}_{x_2})_{\#}\rmR_{0,T}$.
\end{proof}

The Schr\"odinger system~\eqref{SSp} is particularly useful when $f^T$ and $g^T$ are regular enough. Under~\ref{H1} and~\ref{H4} they inherit the regularity (smoothness and integrability) of the densities of $\mu$, $\nu$ respectively. This follows from the identities
\begin{equation*}
    \frac{\De \mu}{\De\frm_X} = f^T \int_{\R^d} P_T g^T \,\De\frm_V\,,\qquad\frac{\De \nu}{\De\frm_X} = g^T \int_{\R^d} P^\ast_T f^T \,\De\frm_V \,,
\end{equation*}
and since $P^\ast_T f^T$ and $P_T g^T$ are smooth and positive (as a result of the lower bound~\eqref{boundlogappendice}).
Moreover, arguing exactly as in Lemma 2.1 in~\cite{conforti2021formula}, owing to the lower bound in~\eqref{boundlogappendice}, and the continuity of $p_T$, we have that there is $c_{T_0} > 0$, (possibly depending on $\mu$ and $\nu$)  such that for all $T\geq T_0$
\begin{equation*}
\|f^T\|_{L^\infty(\frm)} \|g^T\|_{L^1(\frm)} \leq c_{T_0} \Big\|\frac{\De \mu}{\De \frm_X}\Big\|_{L^\infty(\frm)},\quad \|f^T\|_{L^1(\frm)} \|g^T\|_{L^\infty(\frm)} \leq c_{T_0} \Big\|\frac{\De \nu}{\De \frm_X}\Big\|_{L^\infty(\frm)}.
\end{equation*}
These bounds are pivotal to prove that $f^T\to \De \mu/\De \frm_X$ and $g^T\to \De \nu/\De \frm_X$ as $T\to\infty$ in $L^p(\frm)$ for all $p\in [1,\infty)$ akin to what is done in Lemma 3.6 of~\cite{conforti2021formula}.

To ensure that $f^T$, $g^T$ are in $L^\infty(\frm)$  and with compact support, we work under assumption~\ref{H1} and~\ref{H4} for the rest of the section.
With the help of the forward and adjoint semigroup, and~\eqref{fg} we can write
\begin{equation}\label{SS}
 \begin{cases}
  \mu^T_0=f^T\,P_T g^T\,\frm\,,\\
  \mu^T_T=g^T\,P^*_T f^T\,\frm\,,
 \end{cases}
\end{equation}
where we recall that $\mu^T_t = (X_t,\,V_t)_\# \rmP^T$, with $\rmP^T$ being optimal for~\ref{KSPd}.
Furthermore, if we set,
\begin{equation*}
	f_t^T := P^\ast_t f^T\qquad\text{and}\qquad g_t^T := P_{T-t} g^T\,,
\end{equation*}
then $\mu_t^T$, $t\in [0,T]$, can be represented as
\begin{equation}\label{eq: fg dec semigroups}
	\De \mu_t^T = f_t^T g_t^T \De \frm\,.
\end{equation}
It is also immediate to check that it holds
\begin{equation}\label{pde fg}
 \begin{cases}
  \partial_t f^T_t=L^* f^T_t\\
  \partial_t g^T_t=-L g^T_t
 \end{cases}
\quad\text{and}\quad
\begin{cases}
  \partial_t \log f^T_t=L^*\log f^T_t+\Gamma(\log f^T_t)\\
  \partial_t\log g^T_t=-L \log g^T_t-\Gamma(\log g^T_t)\,,
\end{cases}
\end{equation}
where $\Gamma(h)=\gamma\abs{\nabla_v h}^2$ is the \emph{carr\'e du champ} operator associated to the generator $L$.

\medskip

The $fg$-decomposition gives us a nice representation formula for the relative entropy along the entropic interpolation $(\mu^T_t)_{t\in[0,T]}$. Indeed, if we introduce the functions
\begin{equation*}
 h_f^T(t)\coloneqq \int_{\R^{2d}}\log f^T_t\,\rho^T_t\,\De \frm\quad\text{and}\quad h_b^T(t)\coloneqq \int_{\R^{2d}}\log g^T_t\,\rho^T_t\,\De \frm\quad\forall t\in[0,T]\,,
\end{equation*}
then it easily follows that
\begin{equation}\label{entropia with h}
 \cH(\mu^T_t|\frm)=h_f^T(t)+h_b^T(t),\quad\forall t\in[0,T]\,.
\end{equation}
Moreover, we have 
\begin{equation}\label{derivata h}
 \partial_t h_f^T(t)=-\int_{\R^{2d}} \Gamma(\log f^T_t)\rho_t^T\,\De\frm\quad\text{and}\quad \partial_t h_b^T(t)=\int_{\R^{2d}} \Gamma(\log g^T_t)\rho_t^T\,\De\frm\,.
\end{equation}
For a proof of~\eqref{derivata h} we refer to Lemma 3.8 in \cite{conforti2019second} where the classical setting is studied, the only difference in the kinetic setting being that the operator that acts on $f^T_t$ should be replaced with $L^*$, since $L$ is not self-adjoint.

\medskip

In addition to \eqref{entropia with h}, the $fg$-decomposition gives the following representation for the kinetic entropic cost
\begin{equation}\label{costo con h}\begin{aligned}
 \cC_T(\mu,\nu)=&\cH(\mu^{T}|\rmR_{0,T})=\RXexpect{\rho^{T}\log \rho^{T}}
 \\=&\int_{\R^{2d}}\log f^T\,\rho^T_0\,\De\frm+\int_{\R^{2d}}\log g^T\,\rho^T_T\,\De\frm=h_f^T(0)+h_b^T(T)\,.
\end{aligned}\end{equation}
As a byproduct of~\eqref{entropia with h},~\eqref{derivata h}, and~\eqref{costo con h}, we get the identities\footnote{Let us point out that when considering the optimal solution $\rmP^T\in\cP(\Omega)$ in \eqref{eq: KSPd stoch control}, the optimal control is given by $\alpha^\rmP_t=2\gamma\,\nabla_v\log g_t^T(X_t)$ and the stochastic control formulation \eqref{eq: KSPd stoch control} reads as the first identity in \eqref{rappr costo}.}
\begin{equation}\label{rappr costo}
 \begin{aligned}\cC_T(\mu,\nu)=&\cH(\mu^T_0|\frm)+\int_0^T \int_{\R^{2d}}\Gamma(\log g^T_t)\rho^T_t\,\De\frm\,\De t\,,\\
  \cC_T(\mu,\nu)=&\cH(\mu^T_T|\frm)+\int_0^T \int_{\R^{2d}}\Gamma(\log f^T_t)\rho^T_t\,\De\frm\,\De t\,.
 \end{aligned}
\end{equation}
A straightforward consequence of the previous identities is the following 

\begin{lemma}\label{lemmarapprcosto}
 Under the assumptions \ref{H1}, \ref{H4}, for any $t\in[0,T]$ it holds
 \begin{equation}\label{identitycosthalftime}
  \begin{aligned}\cC_T(\mu,\nu)=\,\cH(\mu^T_0|\frm)+\cH(\mu^T_T|\frm)&\,+\int_0^{t} \int_{\R^{2d}}\Gamma(\log g^T_s)\rho^T_s\,\De\frm\,\De s\\
  &\,+\int_{t}^T \int_{\R^{2d}}\Gamma(\log f^T_s)\rho^T_s\,\De\frm\,\De s-\cH\left(\mu^T_{t}|\frm\right)\,.
\end{aligned} \end{equation}
\end{lemma}
\begin{proof}
 From~\eqref{rappr costo} we can write
 \[\cC_T(\mu,\nu)=\cH(\mu^T_0|\frm)+\int_0^{t} \int_{\R^{2d}}\Gamma(\log g^T_s)\rho^T_s\,\De\frm\,\De s+\int_{t}^T \int_{\R^{2d}}\Gamma(\log g^T_s)\rho^T_s\,\De\frm\,\De s\,.\]
Applying the identities~\eqref{derivata h} we obtain that the last summand equals
\begin{equation*}
\begin{aligned}
\int_{t}^T \int_{\R^{2d}}\Gamma(\log g^T_s)\rho^T_s\,\De\frm\,\De s=&\,\int_{t}^T \int_{\R^{2d}}\Gamma(\log f^T_s)\rho^T_s\,\De\frm\,\De s+\int_{t}^T \partial_s h_b^T(s)+\partial_s h_f^T(s)\,\De s\\
=&\, \int_{t}^T \int_{\R^{2d}}\Gamma(\log f^T_s)\rho^T_s\,\De\frm\,\De s+\int_{t}^T \partial_s \cH(\mu_s^T|\frm)\,\De s\\
=&\, \int_{t}^T \int_{\R^{2d}}\Gamma(\log f^T_s)\rho^T_s\,\De\frm\,\De s+\cH(\mu^T_T|\frm)-\cH\left(\mu_{t}^T|\frm\right)\,,
\end{aligned}
\end{equation*}
and we reach our conclusion.
\end{proof}

\subsection{Corrector estimates and proof of Theorem \ref{teo entropic turnpike}}\label{sec: corr est} Throughout  we assume~\ref{H1},~\ref{H2} and~\ref{H4} to be true and we will point out whenever the latter can be relaxed to \ref{H3}. Let us start by defining  a few key objects whose behaviour will help us in controlling the convergence rates for the turnpike property.

We define the \emph{correctors} as the functions $\varphi^T,\psi^T\colon [0,T]\to\R$ given by
 \begin{equation}\label{defcorrectors}
	\varphi^T(s) \coloneqq \int_{\R^{2d}}\abs{\nabla \log f_s^T}^2_{N^{-1}} \rho_s^T \,\De \frm \quad\mbox{and}\quad \psi^T(s) \coloneqq \int_{\R^{2d}} \abs{\nabla \log g_s^T}^2_{M^{-1}} \rho_s^T \,\De \frm\,,
\end{equation}
where $M,\,N \in \R^{2d\times 2d}$ are the matrices appearing in Proposition~\ref{prop:contraction}.
Let us also note that by the $fg$-decomposition it follows
$\cI(\mu^T_s)\lesssim \varphi^T(s)+\psi^T(s)$.

With the next lemma we show that the contraction properties introduced in the previous section translate into an exponentially fast contraction for $\varphi^T$ and $\psi^T$.

\begin{lemma}\label{eq:correctors} Under~\ref{H1},~\ref{H2} and~\ref{H4}, for any $0<t \leq s \leq T$ it holds
\begin{equation}\label{correctors contraction}
	\varphi^T(s) \leq \varphi^T(t) e^{-2\kappa (s-t)}\quad\mbox{and}\quad \psi^T(T-s) \leq \psi^T(T-t) e^{-2\kappa (s-t)}\,.
\end{equation}
\end{lemma}
\begin{proof} By definition $f_s^T = P^*_{s-t} f_t^T$ and thus
\begin{equation}
	\varphi^T(s) = \int_{\R^{2d}} \abs{\nabla \log f_s^T}^2_{N^{-1}} \rho_s^T \,\De \frm = \int_{\R^{2d}} \abs{\nabla P^*_{s-t} f_t^T}^2_{N^{-1}} (P^*_{s-t} f_t^T)^{-1} g_s^T \,\De \frm
\end{equation}
An application of the gradient estimate \eqref{contrsemadj} and Cauchy-Schwartz inequality yields
\begin{equation}
	\begin{aligned}
		\varphi^T(s)&\overset{\eqref{contrsemadj}}{\leq}  e^{-2\kappa (s-t)}\int_{\R^{2d}} \left(P_{s-t}^*\abs{\nabla f_t^T}_{N^{-1}}\right)^2 (P_{s-t}^* f_t^T)^{-1} P_{T-s}g^T \,\De \frm \\
		&\leq  e^{-2\kappa (s-t)}\int_{\R^{2d}} P_{s-t}^*\bigg(\frac{\abs{\nabla f_t^T}^2_{N^{-1}}}{f_t^T}\bigg)  P_{T-s}g^T \,\De \frm\\
		& = e^{-2\kappa (s-t)}\int_{\R^{2d}} \abs{\nabla \log f_t^T}^2_{N^{-1}}  \rho^T_t \,\De \frm \leq e^{-2\kappa  (s-t)} \varphi^T(t),
	\end{aligned}
\end{equation}
which concludes the proof for the first inequality. The analogous inequality for $\psi^T$ runs as above by using inequality \eqref{contrsem} for the semigroup $(P_t)_{t\in[0,T]}$.
\end{proof}

\begin{prop}\label{gio:prop:T/2}
Grant~\ref{H1},~\ref{H2} and~\ref{H4}. There exists $C_{d,\alpha,\beta,\gamma}>0$ such that for any $0<\delta\leq 1$  and for any $t\in[\delta,\,T]$, as soon as $T>\tempo$, it holds
\begin{equation}\label{gio:halfcorrectors}
	\varphi^T(t) \lesssim \,e^{-2\kappa t}\,\left[\cI\left(\mu^T_\delta\right)+\cI\left(
	\mu^T_{T-\delta}\right)\right]\quad\mbox{and}\quad \psi^T(T-t)\lesssim \,e^{-2 \kappa t}\,\left[\cI\left(\mu^T_\delta\right)+\cI\left(
	\mu^T_{T-\delta}\right)\right]\,.
\end{equation}
\end{prop}
\begin{proof}
Without loss of generalities we may assume $\cI\left(\mu^T_\delta\right)$ and $\cI\left(
	\mu^T_{T-\delta}\right)$ to be finite, otherwise the above bounds are trivial. From Lemma \ref{eq:correctors} and the $fg$-decomposition of $\rho_t^T = f_t^T g_t^T$ we know that 
  \begin{align*}
\varphi^T(T-\delta) \leq&\, e^{-2 \kappa\, T+4\kappa\delta}  \varphi^T(\delta)\\
      =&\, e^{-2 \kappa\, T+4\kappa\delta}\int \abs{\nabla \log \rho^T_{\delta}-\nabla \log  g^T_\delta}^2_{N^{-1}} \De \mu^T_{\delta} 
      \\\lesssim &\, e^{-2 \kappa\, T+4\kappa\delta}\cI\left(\mu^T_\delta\right) + e^{-2 \kappa\, T+4\kappa\delta} \psi^T(\delta)\\
      \lesssim & \,e^{-2 \kappa\, T+4\kappa\delta} \cI\left(\mu^T_\delta\right) +  e^{-4 \kappa\, T+8\kappa\delta} \psi^T(T-\delta)\,.
  \end{align*}
  Using the basic inequality $|a-b|^2\geq a^2/2-b^2$ we obtain
  \begin{equation*}
    \varphi^T(T-\delta)=
       \int_{\R^{2d}} \abs{\nabla \log  g^T_{T-\delta} - \nabla \log \rho^T_{T-\delta}}_{N^{-1}}^2\De\mu^T_{T-\delta} 
       \gtrsim  \psi^T(T-\delta)- 2\cI\left(
	\mu^T_{T-\delta}\right)\,.
  \end{equation*}
  As a result, we get
  \begin{equation*}
      \psi^T(T-\delta)- 2\cI\left(
	\mu^T_{T-\delta}\right)\lesssim  e^{-2 \kappa\, T+4\kappa\delta} \cI\left(\mu^T_\delta\right) +   e^{-4 \kappa\, T+8\kappa\delta} \psi^T(T-\delta)\,.
  \end{equation*}
  Therefore, as soon as $T>\frac{1}{\kappa}\log C_{d,\alpha,\beta,\gamma}+2\delta $  for some constant $C_{d,\alpha,\beta,\gamma}>0$, we find
   \begin{equation*}  \psi^T(T-\delta) \lesssim \cI\left(\mu^T_\delta\right)+\cI\left(
	\mu^T_{T-\delta}\right)\,.
   \end{equation*}
   Plugging this bound into the contraction estimate \eqref{correctors contraction} gives the second inequality in \eqref{gio:halfcorrectors} for any $t\in[\delta,T]$. The first inequality is obtained by exchanging the roles of $ \varphi$ and $\psi$ in the above discussion.
\end{proof}

\begin{prop}\label{prop: regularizing effect} Assume~\ref{H1},~\ref{H2} and~\ref{H3}. Let $0<\delta\leq 1$ be fixed. Then, for all $t\in [\delta, T-\delta]$
\begin{equation}\label{Ibound}
    \cI(\mu^T_t) \lesssim \delta^{-3} \Big(\cC_T(\mu,\nu) - \cH(\mu^T_t\mid \frm)\Big).
\end{equation}
\end{prop}
\begin{proof} Let us first work under \ref{H4}.
We claim that for any $t\in[\delta,T-\delta]$ it holds
\begin{equation}\label{bound:after:guillin}
  \abs{\nabla P_{T-t} g^T}^2\lesssim \delta^{-3}\,\Bigl[P_{T-t}(g^T\,\log g^T)-(P_{T-t}\,g^T)\log(P_{T-t}\,g^T)\Bigr] \,P_{T-t}\,g^T\,.
 \end{equation}
 Indeed by applying Corollary 3.2 in~\cite{GuillinWang} to any directional derivative we have
 \begin{align*}
      \abs{\partial_{x_i} P_{T-t} g^T}^2\leq 4\,\inf_{s\in(0,T-t]} \Psi_s(1,0)\Bigl[P_{T-t}(g^T\,\log g^T)-(P_{T-t}\,g^T)\log(P_{T-t}\,g^T)\Bigr] \,P_{T-t}\,g^T\,,\\
      \abs{\partial_{v_i} P_{T-t}g^T}^2\leq 4\,\inf_{s\in(0,T-t]} \Psi_s(0,1)\Bigl[P_{T-t}(g^T\,\log g^T)-(P_{T-t}\,g^T)\log(P_{T-t}\,g^T)\Bigr] \,P_{T-t}\,g^T\,,
 \end{align*}
 where $\Psi_s(a,b)$ is defined for any $a,b>0$ as the quantity
 \[\Psi_s(a,b)\coloneqq\frac{1}{2\gamma} s \left[a\left(\frac{6}{s^2}+\beta+\frac{3\gamma}{2s}\right)+b\left(\frac{4}{s}+\frac{4\beta}{27}s+\gamma\right)\right]^2\,.\]
 By considering $s=\delta\in(0,1]$ we can bound the above RHS with $\delta^{-3}$, up to a multiplicative constant. Particularly this yields  \eqref{bound:after:guillin}. Similarly one can prove that it holds
\begin{equation*}
  \abs{\nabla P^\ast_{t} f^T}^2\lesssim \delta^{-3}\,\Bigl[P^\ast_{t}(f^T\,\log f^T)-(P^\ast_{t}\,f^T)\log(P^\ast_{t}\,f^T)\Bigr] \,P^\ast_{t}\,f^T\,.
 \end{equation*}
 Therefore because of the $fg$-decomposition \eqref{eq: fg dec semigroups} we obtain
\begin{equation*}
\begin{aligned}
    \cI(\mu^T_t) &\leq 2 \int_{\R^{2d}}\left[ \frac{\abs{\nabla P_{T-t} g^T}^2}{P_{T-t} g^T} P^\ast_t f^T + \frac{\abs{\nabla P^\ast_{t} f^T}^2}{P^\ast_t f^T} P_{T-t} g^T\right]\,\De \frm\\
    &\lesssim \delta^{-3}\int_{\R^{2d}} \biggl\{\Bigl[P_{T-t}(g^T\,\log g^T)-(P_{T-t}\,g^T)\log(P_{T-t}\,g^T)\Bigr] P_t^\ast f^T \\ &\qquad\qquad\qquad\qquad\qquad+ \Bigl[P^\ast_{t}(f^T\,\log f^T)-(P^\ast_{t}\,f^T)\log(P^\ast_{t}\,f^T)\Bigr] P_{T-t} g^T\biggr\} \,\De\frm
\end{aligned}
\end{equation*}
By integration by parts and \eqref{eq: fg dec semigroups} this last displacement equals
\begin{equation*}
    \delta^{-3} \bigg(\int_{\R^{d}} \log g^T \,\De \nu + \int_{\R^d}\log f^T \,\De \mu  - \int_{\R^{2d}} \log \rho^T_t\,\rho_t^T \,\De\frm \bigg)\,,
\end{equation*}
and the thesis follows in view of~\eqref{costo con h}.

Now let us just assume \ref{H3}.
Firstly, define the probability measure $q^T_n$ as the measure whose $\rmR_{0,T}$-density is given by
 \begin{equation}\label{approx of optimizer}\frac{\De q^T_n}{\De \rmR_{0,T}}\coloneqq \left(\rho^{T}\wedge n\right)\,\frac{\IND_{K_n}}{C_n}\,,\end{equation}
 where $(K_n)_{n\in\N}$ is an increasing sequence of compact sets in $\R^{4d}$ and $C_n$ is the normalising constant. Then, by applying Lemma \ref{n:lemma} we know that the marginals $\mu^n\coloneqq(\mathrm{proj}_{x_1})_{\#}q^T_n$ and $\nu^n\coloneqq(\mathrm{proj}_{x_2})_{\#}q^T_n$ satisfy \ref{H4} and by means of Proposition \ref{full:approx} and Corollary \ref{full:cor:approx} it follows that there exists a unique minimizer $\mu^{n,T}\in \cP(\Omega)$ for \ref{KSP}  with marginals $\mu^n,\,\nu^n$, and as soon as $n$ diverges it holds
\begin{equation}\label{n:conv:new}\mu^{n,T}_t\weakto\mu^T_t\qquad\text{ and }\qquad\cC_T\left(\mu^n,\nu^n\right)\to\cC_T\left(\mu,\nu\right)\,.\end{equation}
Then, the thesis in the general case follows from the one under \ref{H4} and the lower semicontinuity of $\cI(\cdot)$ and $\cH(\cdot|\frm)$.
\end{proof}

As a byproduct of Proposition \ref{gio:prop:T/2} and Proposition \ref{prop: regularizing effect} we get

\begin{corollary}\label{prop:T/2}
Under~\ref{H1},~\ref{H2} and~\ref{H4},  there exists $C_{d,\alpha,\beta,\gamma}>0$ such that for any $0<\delta\leq1$ and $t\in[\delta,\,T]$, as soon as $T>\tempo$, it holds
\begin{equation}\label{halfcorrectors}
	\varphi^T(t) \lesssim \delta^{-3}\,e^{-2\kappa t}\,\cC_T(\mu,\nu)\quad\mbox{and}\quad  \psi^T(T-t)\lesssim \delta^{-3}\,e^{-2 \kappa t}\,\cC_T(\mu,\nu)\,.
\end{equation}
\end{corollary}

\begin{proof}[Proof of Theorem \ref{teo entropic turnpike}] We start proving the result under~\ref{H4}.
Since $\cI(\mu^T_t)\lesssim \varphi^T(t)+\psi^T(t)$, the first inequality (cf.\ \eqref{fisher eq}) is an immediate consequence of Corollary~\ref{prop:T/2}. The relative entropy bound (cf.\ \eqref{entropic eq}) follows from the first one by means of~\eqref{LS}. In order to extend  \eqref{fisher eq} and  \eqref{entropic eq} to \ref{H3}, it is enough to consider the approximation of the optimizer (cf. \eqref{approx of optimizer} and \eqref{n:conv:new}) together with the lower semicontinuity of $\cI(\cdot)$ and $\cH(\cdot|\frm)$.
Finally, \eqref{bis:entropic eq} follows from \eqref{entropic eq} by means of Lemma \ref{Talagrand4} below.
\end{proof}

\subsection{Long-time behaviour of the kinetic entropic cost}
Throughout the whole section we will always assume~\ref{H1} and~\ref{H3} to be true. Let us remark that \ref{H1} implies that $\frm\in\cP_2(\R^{2d})$.
 In what follows we are going to prove Theorem \ref{long-time cost}, but first we need some preparation. The first two claims in Theorem \ref{long-time cost} will be proved via a $\Gamma$-convergence approach (cf. Proposition \ref{propgammaconv} and Lemma \ref{gammalemmacoercive}) similar to the one used in \cite{conforti2021formula} for the classical Schr\"odinger problem. The main difference with \cite{conforti2021formula} is the lack of compactness for the set $\Pi_X(\mu,\nu)$. On the other hand, the proof of the latter two bounds in Theorem \ref{long-time cost} will rely on the corrector estimates given in Section \ref{sec: corr est}.

\begin{prop}[A $\Gamma$-convergence result]\label{propgammaconv}
Let $(T_n)_{n\in\N}$ be a sequence of positive real numbers converging to $\oo$, and for each $n\in\N$ consider the functional $\cH\left(\cdot\mid \rmR_{0,T_n}\right)$ defined on $\Pi_X(\mu,\,\nu)$ endowed with the weak topology. Then
 \[\Gamma-\lim_{n\to\oo}\cH(\cdot|\rmR_{0,T_n})=\cH(\cdot|\frm\otimes\frm)\,.\]
\end{prop}
\begin{proof}
($\Gamma$-convergence lower bound inequality) We prove that for any sequence $\left(q_n\right)_{n\in\N}\subset \Pi_X\left(\mu,\,\nu\right)$ that converges weakly to some $q\in\Pi_X(\mu,\,\nu)$ 
\begin{equation}\label{gammaliminf}
  \liminf_{n\to\oo} \cH\left(q_n\mid\rmR_{0,T_n}\right)\geq \cH\left(q\mid \mathfrak{m}\otimes\mathfrak{m}\right)\,.
 \end{equation}
 Note that since $P_t$ is strongly mixing \cite[Theorem 11.14]{DaPrato} for any $\psi,\phi\in C_b(\R^{2d})$  it holds
 \begin{equation*} \begin{aligned}&\int_{\R^{2d}}\int_{\R^{2d}}\psi(x,v)\phi(y,w)\De \rmR_{0,T_n}=\int_{\R^{2d}}\psi\,P_{T_n}\phi\,\De \mathfrak{m}\,
 \mapor{n\to \oo}\,\int_{\R^{2d}}\int_{\R^{2d}}\psi(x,v)\phi(y,w)\De\mathfrak{m}\otimes\De\mathfrak{m}\,.
 \end{aligned}\end{equation*}
  From the Portmanteau Theorem, it follows that $\rmR_{0,T_n}\weakto \mathfrak{m}\otimes\mathfrak{m}$. Then \eqref{gammaliminf} follows from the lower semicontinuity of the relative entropy.

($\Gamma$-convergence upper bound inequality) We prove that  for any $q\in\Pi_X(\mu,\,\nu)$ it holds 
\begin{equation}\label{gammalimsupparticolare}
 \limsup_{n\to\oo} \cH\left(q\mid\rmR_{0,T_{n}}\right)\leq \cH\left(q\mid\mathfrak{m}\otimes\mathfrak{m}\right)\,.
\end{equation}
We may assume $\cH(q|\frm\otimes\frm)<\oo$ otherwise the above inequality is trivial. Note that this implies $q\in\cP_2(\R^{4d})$ since $\mu,\nu\in\cP_2(\R^d)$ (cf. \eqref{secondmomentbound}) while
\begin{align*}\int_{\R^d}\abs{v}^2\De (\mathrm{proj}_{v_1})_{\#}q\leq 2\int_{\R^d}\abs{v}^2\De\frm_V+2\,\cW_2((\mathrm{proj}_{v_1})_{\#}q,\frm_V)^2\\
\overset{\eqref{talagrand}}{\lesssim} 1+ \,\cH((\mathrm{proj}_{v_1})_{\#}q|\frm_V)\leq 1+\cH(q|\frm\otimes\frm)<\oo\,,\end{align*}
and similarly for the measure $(\mathrm{proj}_{v_2})_{\#}q$.
Then we have
\[\cH\left(q\mid\rmR_{0,T_n}\right)=\cH\left(q\mid\mathfrak{m}\otimes\mathfrak{m}\right)-\int_{\R^{2d}\times\R^{2d}} \log p_{T_n}\bigl((x,v),\,(y,w)\bigr)\,\De q\,,\]
 Thanks to the lower bound given in Lemma \ref{lemmaappendicebound} and the fact that $q\in\cP_2(\R^{4d})$, we can apply Fatou's Lemma and get
  \begin{equation}\label{appolemmagamma5}\limsup_{n\to\oo} \cH\left(q\mid\rmR_{0,T_n}\right)\leq\cH\left(q\mid\mathfrak{m}\otimes\mathfrak{m}\right)-\int_{\R^{2d}\times\R^{2d}} \liminf_{n\to\oo}  \log p_{T_n}\bigl((x,v),\,(y,w)\bigr)\,\De q\,.
  \end{equation}
 Now, for all $t>0$ and for all $(x,v),\,(y,w)\in\R^{2d}$
 \[p_{T_n}\bigl((x,v),\,(y,w)\bigr)= P_{T_n-t}\left(p_t \bigl(\cdot\,,\,(y,w)\bigr)\right)(x,v)\,.\]
 For any $M>0$, we introduce the function $p_t^M(\cdot\,,\,(y,w))\coloneqq p_t (\cdot\,,\,(y,w))\wedge M\in C_b(\R^{2d})$. Then, since $P_T$ is strongly mixing (cf. \cite[Theorem 11.14]{DaPrato}), we get
 \begin{equation*}\begin{aligned}
 p_{T_n}&\bigl((x,v),\,(y,w)\bigr)\geq P_{T_n-t}\left(p_t^M \bigl(\cdot\,,\,(y,w)\bigr)\right)(x,v)\,\mapor{n\to\oo}\,\int_{\R^{2d}}p_t^M \bigl((x,v),(y,w)\bigr)\De \mathfrak{m}(x,v)\,.
 \end{aligned}\end{equation*}

 Taking the limit as $M\to\oo$, by dominated convergence we get that
 \[\int_{\R^{2d}}p_t^M \bigl((x,v),(y,w)\bigr)\De \mathfrak{m}(x,v)\,\mapor{M\to\oo} \int_{\R^{2d}}p_t \bigl ((x,v),(y,w)\bigr)\De \mathfrak{m}(x,v)=1\,,\]
 where the last equality follows from \eqref{physrever}.
 Therefore it holds
 \[\liminf_{n\to\oo}\log p_{T_{n}}\bigl((x,v),\,(y,w)\bigr)\geq 0\,,\quad\mathfrak{m}\otimes\mathfrak{m}-a.s.\]
which, together with  $q\ll\mathfrak{m}\otimes\mathfrak{m}$ (since $\cH\left(q \mid\mathfrak{m}\otimes\mathfrak{m}\right)<\oo$), leads to
 \[\liminf_{n\to\oo}\log p_{T_{n}}\bigl((x,v),\,(y,w)\bigr)\geq 0\,,\qquad q\text{-a.s.}\]
Therefore, from \eqref{appolemmagamma5} we get inequality \eqref{gammalimsupparticolare}. The desired $\Gamma$-convergence follows as a byproduct of \eqref{gammaliminf} and \eqref{gammalimsupparticolare}.
\end{proof}
Even though here we are just interested in the $\Gamma$-convergence (as introduced by De Giorgi) on $\Pi_X(\mu,\nu)$ equipped with the weak topology, the previous result is actually stronger: indeed we have actually proven the Mosco convergence of the functional $\cH\left(\cdot\mid \rmR_{0,T_n}\right)$  since we have considered a constant sequence $q_n=q$ for the upper bound inequality.

\begin{lemma}[Equicoerciveness]\label{gammalemmacoercive}
 The family $\left\{\cH(\cdot\mid \rmR_{0,T_n})\colon \Pi_X(\mu,\nu)\to[0,\oo]\right\}_{n\in\N}$ is equicoercive, i.e. for any $h\in\R$ there exists a (weakly) compact subset $K_h\subset\Pi_X(\mu,\,\nu)$ such that
 \[\bigl\{q\in\Pi_X\left(\mu,\,\nu\right) \text{ s.t. }\cH(q\mid\rmR_{0,T_n})\leq h\bigr\}\subseteq K_h\quad\forall n\in\N\,.\]
\end{lemma}
\begin{proof}
 Since $(\rmR_{0,T_n})_{n\in\N}$ is tight, a proof of this result is obtained by following the same argument given in \cite[Lemma 1.4.3c]{DupuisEllis}.
\end{proof}

The next two results are a consequence of the corrector estimates of Section \ref{sec: corr est}. In the first one we consider the long-time behaviour of the marginals of the solution to \ref{KSP} at times $t=0,T$; in the second one we give a bound for the entropic cost, uniformly in time.

\begin{theorem}\label{exp marginal entropies}
 Under assumptions \ref{H1},\ref{H2} and \ref{H4} there exists a positive constant $C_{d,\alpha,\beta,\gamma}$ such that for any $0<\delta\leq1$ and $T>\tempo$ it holds
 \begin{equation}\label{expmarginals}\begin{aligned}\abs{\cH(\mu^T_0|\frm)-\cH(\mu|\frm_X)}\leq C_{d,\alpha,\beta,\gamma} \,\delta^{-3}\,\cC_T(\mu,\nu)\,e^{-2\kappa\, T}\,,\\
 \abs{\cH(\mu^T_T|\frm)-\cH(\nu|\frm_X)}\leq C_{d,\alpha,\beta,\gamma} \,\delta^{-3}\,\cC_T(\mu,\nu)\,e^{-2\kappa\, T}\,.\end{aligned}\end{equation}
\end{theorem}
\begin{proof}
 We will prove only the first bound since the second one can be proved similarly. Since $\frac{\De\mu}{\De\frm_X}(\cdot)=\int_{\RD} \rho^T_0(\cdot,v)\De\frm_V(v)$, the log-Sobolev inequality for the Gaussian measure $\frm_V$ gives
 \begin{equation*}
  \begin{aligned}
   \cH(\mu^T_0|\frm)-&\cH(\mu|\frm_X)=
   \int_{\RD}\biggl[\int_{\RD} \rho^T_0\log\rho^T_0-\left(\int_{\RD} \rho^T_0\De\frm_V\right)\log\left(\int_{\RD} \rho^T_0\De\frm_V\right) \,\De\frm_V\biggr]\,\De\frm_X\\
   \leq&\int_{\RD}\biggl[\int_{\RD}\Big|\nabla_v \sqrt{\rho^T_0}\Big|^2\,\De\frm_V\biggr]\,\De\frm_X=\int_{\R^{2d}}\frac{1}{2}\abs{\frac{\nabla_v \rho^T_0}{\rho^T_0}}^2\rho^T_0\,\De\frm\\
   =&\frac{1}{2}\int_{\R^{2d}}\abs{\nabla_v\log (f^TP_Tg^T)}^2\rho^T_0\,\De\frm=\frac{1}{2}\int_{\R^{2d}}\abs{\nabla_v\log f^T+\nabla_v\log P_T g^T}^2\rho^T_0\,\De\frm\\
   =&\frac{1}{2}\int_{\R^{2d}}\abs{\nabla_v\log g_0^T}^2\rho^T_0\,\De\frm\lesssim \psi^T(0)
   \overset{\eqref{halfcorrectors}}{\lesssim}  \,\delta^{-3}\,\cC_T(\mu,\nu)\,e^{-2\kappa\, T},
  \end{aligned}
 \end{equation*}
 where the equality in the last line follows from the fact that $f^T=f^T(x)$ does not depend on the velocity variable. Finally, since ${\left(\mathrm{proj}_x\right)}_{\#}\,\mu^T_0=\mu$ we know that the left hand side term above is positive.
\end{proof}

Since it holds $\cH(\mu|\frm_X)=\cH\left( (X_0)_{\#}\mu^T|(X_0)_{\#}\rmR_{0,T}\right)\leq \cH\left( \mu^T|\rmR_{0,T}\right)=\cC_T(\mu,\nu)$, and similarly $\cH(\nu|\frm_X)\leq\cC_T(\mu,\nu)$, the following lower bound is always true
\begin{equation}\label{eq:lowerbound}
    \cC_T(\mu,\nu)\geq \frac{\cH(\mu|\frm_X)+\cH(\nu|\frm_X)}{2}\,.
\end{equation}
We now  give a corresponding upper bound for sufficiently large times.

\begin{lemma}\label{Talagrand4}
  Under~\ref{H1} and~\ref{H2} there exists a constant $C_{d,\alpha,\beta,\gamma}>0$ such that  for any $0<\delta\leq1$ and $T>(\tempo)\vee\bigl(\tempobis\bigr)$ it holds
\begin{equation}\label{eq:Talagrand4}\cC_T(\mu,\nu)\leq C_{d,\alpha,\beta,\gamma}\, \biggl[\cH(\mu|\frm_X)+\cH(\nu|\frm_X)\biggr]\,.\end{equation}
\end{lemma}
\begin{proof} 
Firstly, let us assume \ref{H4} to hold.
 Owing to the bounds $|\nabla_v\log g^T_s|^2\lesssim |\nabla \log g^T_s|^2_{M^{-1}}$ and $|\nabla_v\log f^T_s|^2\lesssim |\nabla \log f^T_s|^2_{N^{-1}}$, from~\eqref{identitycosthalftime} it follows
 \begin{equation*}
  \begin{aligned}\cC_T(\mu,\nu)
  &\leq \cH(\mu^T_0|\frm)+\cH(\mu^T_T|\frm)+\int_0^{\frac{T}{2}} \psi^T(s)\,\De s+\int_{\frac{T}{2}}^T \varphi^T(s)\,\De s\\
  &\overset{\eqref{expmarginals}}{\lesssim}\cH(\mu|\frm_X)+\cH(\nu|\frm_X)+2\,\delta^{-3}\,\cC_T(\mu,\nu)\,e^{-2\kappa\,T}+\int_0^{\frac{T}{2}} \psi^T(s)\,\De s+\int_{\frac{T}{2}}^T \varphi^T(s)\,\De s\,.
\end{aligned} \end{equation*}

We first consider $\int_{T/2}^T \varphi^T(s)\,\De s$. For any $s\in[T/2,\,T]$ from Corollary \ref{prop:T/2} we have 
\begin{equation}\label{nek1}\int_{\frac{T}{2}}^T \varphi^T(s)\,\De s\lesssim \delta^{-3}\,\cC_T(\mu,\nu)\,\int_{\frac{T}{2}}^T e^{-2\kappa\,s}\,\De s\lesssim \delta^{-3}\,\cC_T(\mu,\nu)\,\bigl(e^{-\kappa\,T}-e^{-2\kappa\,T}\bigr)\,.\end{equation}

By reasoning in the same way, this time by using the fact that $s\in[0,T/2]$, we get
\begin{equation}\label{nek2}\int_0^{\frac{T}{2}} \psi^T(s)\,\De s\lesssim \delta^{-3}\,\cC_T(\mu,\nu)\,\int_0^{\frac{T}{2}} e^{-2\kappa\,(T-s)}\,\De s\lesssim \delta^{-3}\,\cC_T(\mu,\nu)\,\bigl(e^{-\kappa\,T}-e^{-2\kappa\,T}\bigr)\,.\end{equation}
Therefore there exists a positive constant $C_{d,\alpha,\beta,\gamma}$ such that
\begin{equation*}\begin{aligned}\cC_T(\mu,\nu)
\leq C_{d,\alpha,\beta,\gamma}\biggl[ \,\cH(\mu|\frm_X)+\cH(\nu|\frm_X)+\delta^{-3}\,\cC_T(\mu,\nu)\,\bigl(e^{-\kappa\,T}-e^{-2\kappa\,T}\bigr)\biggr]\,,
\end{aligned}\end{equation*}
which yields  our thesis as soon as $T>\tempobis$ for a well chosen  $C_{d,\alpha,\beta,\gamma}>0$.

Now, let us prove the result under \ref{H2}. Firstly, notice that we may assume that $\mu$ and $\nu$ satisfy \ref{H3}, otherwise the bound is trivial. The main idea is defining the probability measures $\mu_n^M$ and $\nu_n^M$ on $\bbR^d$, approximating $\mu$ and $\nu$, as the measures whose $\frm_X$-densities are given by
 \begin{equation}\label{M:n:def}\frac{\De\mu^M_n}{\De \frm_X}\coloneqq \left(\frac{\De\mu}{\De \frm_X}\wedge n\right)\,\frac{\IND_{K_n}}{C^\mu_n}\qquad\text{and}\qquad\frac{\De\nu^M_n}{\De \frm_X}\coloneqq \left(\frac{\De\nu}{\De \frm_X}\wedge n\right)\,\frac{\IND_{K_n}}{C^\nu_n}\,,\end{equation}
 where $(K_n)_{n\in\N}$ is an increasing sequence of compact sets in $\R^d$  and $C_n^{\mu},\,C_n^\nu$  are the normalising constants. Then, $\mu^M_n$ and $\nu_n^M$ satisfy \ref{H4} and by means of \eqref{conv:approx:marginals} it follows
 \begin{equation}\label{M:n:entropie}
 \cH(\mu_n^M|\frm_X) \overset{n\to\oo}{\longrightarrow} \cH(\mu|\frm_X)\qquad\text{and}\qquad \cH(\nu_n^M|\frm_X) \overset{n\to\oo}{\longrightarrow} \cH(\nu|\frm_X)\,.
 \end{equation}
 Owing to Lemma \ref{M:n:lemma} and \eqref{eq:Talagrand4} for the approximated $\mu^M_n,\nu^M_n$, we conclude our proof.
\end{proof}

The results given in Theorem \ref{exp marginal entropies} and Lemma \ref{Talagrand4} will come at hand while proving the exponential convergence in~\eqref{good:exp conv}.

\begin{proof}[Proof of Theorem \ref{long-time cost}]
We start with the proof of~\eqref{longcost} and~\eqref{weakconv}. Firstly note that the unique minimizer in \eqref{starimportance}
is given the probability measure $\mu^\oo\coloneqq(\mu\otimes\frm_V)\otimes(\nu\otimes\frm_V)$. Now let us consider  $(T_n)_{n\in\N}$ to be any diverging sequence of positive real times. Then, from the optimality of $\mu^{T_n}$ it follows
 \begin{equation*}\begin{aligned}
 \limsup_{n\to\oo}\cH(\mu^{T_n}|\rmR_{0,T_{n}})\leq \limsup_{n\to\oo}\cH(\mu^\oo|\rmR_{0,T_{n}})
 \overset{\eqref{gammalimsupparticolare}}{\leq}\cH(\mu^\oo|\mathfrak{m}\otimes\mathfrak{m})=\cH(\mu\mid\frm_X)+\cH(\nu\mid\frm_X),
 \end{aligned}\end{equation*}
 which is finite by our assumptions. Then Lemma \ref{gammalemmacoercive} implies that the subsequence $(\mu^{T_n})_{n\in\N}$ is weakly relatively compact.
 Then, from Proposition \ref{propgammaconv}, the Fundamental theorem of $\Gamma$-convergence \cite[Theorem 2.10]{Braideshandbook}, 
the uniqueness of the minimizer in \eqref{starimportance} and from the metrizability of the weak convergence on $\cP(\R^{4d})$ we deduce \eqref{longcost} and \eqref{weakconv}.

We continue with the proof of \eqref{good:exp conv} and of the entropic Talagrand inequality \eqref{gio:talagrand}.
 We firstly assume \ref{H4} to hold. By~\eqref{identitycosthalftime}  and  owing to $|\nabla_v\log g^T_s|^2\lesssim |\nabla \log g^T_s|^2_{M^{-1}}$ and $|\nabla_v\log f^T_s|^2\lesssim |\nabla \log f^T_s|^2_{N^{-1}}$, we know that
 \[\abs{\cC_T(\mu,\nu)- \cH(\mu^T_0|\frm)-\cH(\mu^T_T|\frm)}\lesssim \cH(\mu^T_{\frac{T}{2}}|\frm)+\int_0^{\frac{T}{2}} \psi^T(s)\,\De s+\int_{\frac{T}{2}}^T \varphi^T(s)\,\De s\,,\]
 and from \eqref{nek1}, \eqref{nek2} and the entropic turnpike \eqref{entropic eq} it follows  
 \[\abs{\cC_T(\mu,\nu)- \cH(\mu^T_0|\frm)-\cH(\mu^T_T|\frm)}\lesssim \delta^{-3}\,e^{-\kappa\,T}\,\cC_T(\mu,\nu)+\delta^{-3}\,\cC_T(\mu,\nu)\,e^{-\kappa\,T}\,.\]
 As a byproduct of the above inequality and Theorem \ref{exp marginal entropies} we get 	\begin{equation}\label{exp conv}
	\begin{aligned}
		\abs{\cC_{T}\left(\mu,\nu\right)- \cH\left(\mu|\frm_X\right) -  \cH\left(\nu|\frm_X\right)}\leq C_{d,\alpha,\beta,\gamma}\,\delta^{-3}\,e^{-\kappa\,T}\,\cC_T(\mu,\nu)\, .
	\end{aligned}\end{equation}	
 
 Let us now assume that $\mu$ and $\nu$ satisfy \ref{H3} only. 
 Firstly, consider the approximating sequence $(\mu^{T}_{n})_{n\in\N}$ of the optimizer (cf. \eqref{approx of optimizer} and \eqref{n:conv:new}). Then, by means of \eqref{exp conv} under \ref{H4} and the lower semicontinuity of the relative entropy  we have
 \begin{equation*}
  \begin{aligned}\cH&\left(\mu|\frm_X\right) +  \cH\left(\nu|\frm_X\right)-
   \cC_{T}\left(\mu,\nu\right)\overset{\eqref{n:conv:new}}{\leq} \liminf_{n\to\oo}\biggl[\cH\left(\mu^n|\frm_X\right) +  \cH\left(\nu^n|\frm_X\right)-
   \cC_{T}\left(\mu^n,\nu^n\right)\biggr]\\
   \overset{\eqref{exp conv}}{\lesssim}&\delta^{-3}e^{-\kappa\,T} \liminf_{n\to\oo}\cC_T(\mu^n,\nu^n)=\delta^{-3}e^{-\kappa\,T}\,\cC_T(\mu,\nu)\overset{\eqref{eq:Talagrand4}}{\lesssim}\,\delta^{-3}\,e^{-\kappa\,T}\,\biggl[ \cH\left(\mu|\frm_X\right)+\cH\left(\nu|\frm_X\right)\biggr]\,.
  \end{aligned}
 \end{equation*}

For the other bound we are going to use the approximation on the marginals (cf. \eqref{M:n:def}). Therefore, let us consider $\mu^M_n,\,\nu^M_n$ such that \ref{H4} holds. Then, from Lemma \ref{M:n:lemma}  and the convergence of the relative entropies in \eqref{M:n:entropie}, we get
 \begin{equation*}
  \begin{aligned} \cC_{T}&\left(\mu,\nu\right)-\cH\left(\mu|\frm_X\right) - \cH\left(\nu|\frm_X\right)\leq\liminf_{n\to\oo}\biggl[\cC_T(\mu^M_n,\nu^M_n)-\,\cH(\mu^M_n|\frm_X)-\cH(\bar\nu_n^M|\frm)\biggr]\\
  \overset{\eqref{exp conv}}{\lesssim}&\,\delta^{-3}\,e^{-\kappa\,T}\, \liminf_{n\to\oo}\cC_T\left(\mu^M_n,\nu^M_n\right)
  \overset{\eqref{eq:Talagrand4}}{\lesssim}\delta^{-3}\,e^{-\kappa\,T}\, \liminf_{n\to\oo}\,\biggl[ \cH\left(\mu^M_n|\frm_X\right)+\cH\left(\nu^M_n|\frm_x\right)\biggr]\\
  \,&\qquad\qquad\qquad\qquad\qquad\qquad\qquad=\,\delta^{-3}\,e^{-\kappa\,T}\,\biggl[ \cH\left(\mu|\frm_X\right)+\cH\left(\nu|\frm_X\right)\biggr]\,.
\end{aligned} \end{equation*}
 \end{proof}

\subsection{Corrector estimates for~\ref{KFSP} and proof of Theorems~\ref{full:teolongcost} and~\ref{full:teo entropic turnpike}}
In this section we collect results in the kinetic-full setting analogous to the ones already presented for~\ref{KSP}. We omit the proofs since the arguments are very similar and do not present any new difficulty with respect to the~\ref{KSP} case.

Let us start by mentioning that also in this case~\ref{KFSP} and~\ref{KFSPd} admit unique solutions $\bar{\mu}^T,\,\bar{\rmP}^T$ with $\bar{\mu}^T= ((X_0,V_0),(X_T,V_T))_\# \bar{\rmP}^T$ which can be decomposed as
 \begin{equation}\label{fg full definition}
  \bar\rho^T(x,v,y,w)\coloneqq\frac{\De \bar \mu^{T}}{\De \rmR_{0,T}}(x,v,y,w)=\bar{f}^T(x,v)\bar{g}^T(y,w)\qquad\rmR_{0,T}\textrm{-a.s.}
 \end{equation}
 where $\bar{f}^T,\,\bar{g}^T$ are two non-negative  measurable functions on $\R^{2d}$ that solve the Schr\"odinger system
 \begin{equation}\label{FSSp}
\begin{cases}
\frac{\De\bar\mu}{\De\frm}(x,v)  =\bar{f}^T(x,v) \,\bbE_\rmR\big[\bar{g}^T(X_T,V_T)|X_0=x,\,V_0=v \big]\,,\\
\frac{\De\bar\nu}{\De\frm}(y,w)  =\bar{g}^T(y,w)\,\bbE_\rmR\big[\bar{f}^T(X_0,V_0)|X_T=y,\,V_T=w \big]\,.
\end{cases}
\end{equation}
Note that in this case $f$ and $g$ are function of  both space and velocity. Moreover, if we define for any $t\in[0,T]$
\begin{equation*}
	\bar{f}^T_t := P^\ast_t \bar{f}^T\qquad\text{and}\qquad \bar{g}_t^T := P_{T-t} \bar{g}^T\,,
\end{equation*}
 then $\bar{\mu}^T_t = (X_t,\,V_t)_\# \bar{\rmP}^T$ can be written  as $
	\bar{\mu}^T_t = \bar{f}^T_t \bar{g}^T_t \frm$ and, similarly to~\eqref{identitycosthalftime}, under~\ref{H1} and~\ref{FH4} it holds that, for any $t\in[0,T]$
 \begin{equation}\label{full:identitycosthalftime}
  \begin{aligned}\cC_T^F\left(\bar\mu,\bar\nu\right)=\,\cH(\bar\mu|\frm)&\,+\cH(\bar\nu|\frm)-\cH(\bar{\mu}^T_{t}|\frm)\\
  &\,+\int_0^{t} \int_{\R^{2d}}\Gamma(\log \bar{g}^T_s)\bar{\rho}^T_s\,\De\frm\,\De s+\int_{t}^T \int_{\R^{2d}}\Gamma(\log \bar{f}^T_s)\bar{\rho}^T_s\,\De\frm\,\De s\,.
\end{aligned} \end{equation}

We can therefore define the \emph{correctors} as the functions $\bar{\varphi}^T,\bar{\psi}^T\colon [0,T]\to\R$ given by
 \begin{equation}
    \bar{\varphi}^T(s) \coloneqq \int_{\R^{2d}}|\nabla \log \bar{f}_s^T|^2_{N^{-1}} \bar{\rho}_s^T \,\De \frm \quad\mbox{and}\quad \bar{\psi}^T(s) \coloneqq \int_{\R^{2d}} |\nabla \log \bar{g}_s^T|^2_{M^{-1}} \bar{\rho}_s^T \,\De \frm\,,
\end{equation}
where $M,\,N\in \R^{2d\times 2d}$ are positive definite symmetric matrices as appearing in Proposition~\ref{prop:contraction}. In the next result we collect all the contraction properties satisfied by the above correctors, which correspond to the ones proven for~\ref{KSP} in Lemma~\ref{eq:correctors}, Proposition~\ref{gio:prop:T/2}, Proposition~\ref{prop: regularizing effect} and Corollary~\ref{prop:T/2}.
\begin{lemma}\label{full:omnicomprensivo:correctors}
Grant~\ref{H1},~\ref{H2},~\ref{FH4} and fix $\delta\in(0,1]$. For any $0 < t \leq s \leq T$ it holds
\begin{equation*}
	\bar{\varphi}^T(s) \leq \bar{\varphi}^T(t) e^{-2\kappa\, (s-t)}\quad\mbox{and}\quad \bar{\psi}^T(T-s) \leq \bar{\psi}^T(T-t) e^{-2\kappa\, (s-t)}\qquad\forall\,0 < t \leq s \leq T\,.
\end{equation*}
Moreover, for any fixed $\delta\in(0,1]$ as soon as $T>\tempo$ the followings hold true
\begin{subequations}
\begin{equation}
	\bar\varphi^T(t) \lesssim \,e^{-2\kappa t}\,\left[\cI\left(\bar\mu^T_\delta\right)+\cI\left(
	\bar\mu^T_{T-\delta}\right)\right]\mbox{ and }\, \bar\psi^T(T-t)\lesssim \,e^{-2 \kappa t}\,\left[\cI\left(\bar\mu^T_\delta\right)+\cI\left(
	\bar\mu^T_{T-\delta}\right)\right]\,\forall t\in[\delta,T]\,,
	\end{equation}
	\begin{equation}
	\cI(\bar \mu^T_t) \lesssim \delta^{-3} \Big(\cC^F_T(\bar\mu,\bar\nu) - \cH(\bar\mu^T_t\mid \frm)\Big)\quad\forall t\in[\delta, T-\delta]\,,    
	\end{equation}
	\begin{equation}\label{full:halfcorrectors}
\bar{\varphi}^T(t) \lesssim \delta^{-3}\,e^{-2\kappa \,t}\,\cC_T^F\left(\bar\mu,\bar\nu\right)\text{ and }\,  \bar{\psi}^T(T-t)\lesssim \delta^{-3}\,e^{-2\kappa \,t}\,\cC_T^F\left(\bar\mu,\bar\nu\right)\quad\forall t\in[\delta,T]\,.
\end{equation}
	\end{subequations}
\end{lemma}
\begin{proof}[Proof of Theorem~\ref{full:teo entropic turnpike}]
The proof of the result under~\ref{FH4} follows the same reasoning presented in the first part of the proof of Theorem~\ref{teo entropic turnpike} and for this reason is omitted. An approximating argument akin to the one in the proof of Theorem~\ref{teo entropic turnpike}, this time considering the full marginals $\bar\mu^n,\bar\nu^n$ and the corresponding \ref{KFSP}, gives \begin{equation*}\bar\mu^{n,T}_t\weakto\bar\mu^T_t\qquad\text{ and }\qquad\cC_T^F\left(\bar\mu^n,\bar\nu^n\right)\to\cC_T^F\left(\bar\mu,\bar\nu\right)\,.\end{equation*}
Therefore the first two bounds follow from the lower semicontinuity of $\cI(\cdot)$ and $\cH(\cdot|\frm)$.
Finally, \eqref{bis:full:entropic eq} follows from \eqref{full:entropic eq} by means of \eqref{eq:full:Talagrand4} presented below.
\end{proof}

\begin{proof}[Proof of Theorem \ref{full:teolongcost}]
The proof of~\eqref{full:longcost} and~\eqref{full:weakconv} runs similarly to the one given above in the kinetic setting and for this reason it is omitted. The main difference is that in this case the equicoerciveness is not needed since we have the weak compactness of $\Pi\left(\bar\mu,\bar\nu\right)$.

 We now discuss~\eqref{good:full:exp conv} and \eqref{gio:full:talagrand}. With similar arguments as for~\ref{KSP} one can show that under~\ref{H1} and~\ref{H2}, there exists a constant $C_{d,\alpha,\beta,\gamma}>0$ such that  for any $0<\delta\leq1$ and $T>(\tempo)\vee\bigl(\tempobis\bigr)$ it holds
\begin{equation}\label{eq:full:Talagrand4}\cC_T^F\left(\bar\mu,\bar\nu\right)\leq C_{d,\alpha,\beta,\gamma}\, \biggl[\cH(\bar\mu|\frm)+\cH(\bar\nu|\frm)\biggr]\,.\end{equation}
Further, by means of
\eqref{full:identitycosthalftime}, the corrector estimates \eqref{full:halfcorrectors} and the turnpike estimate~\eqref{full:entropic eq}, at least under~\ref{FH4}, it follows  that  for any $0<\delta\leq1$,  as soon as $T>(\tempo)\vee\tempobis$,	\begin{equation*}
	    \abs{\cC^F_{T}\left(\bar\mu,\bar\nu\right)- \cH\left(\bar\mu|\frm\right) -  \cH\left(\bar\nu|\frm\right)} \leq C_{d,\alpha,\beta,\gamma}\,\delta^{-3}\,e^{-\kappa\,T}\,\biggl[ \cH\left(\bar\mu|\frm\right)+\cH\left(\bar\nu|\frm\right)\biggr]\, ,
\end{equation*}	
and from this immediately deduce~\eqref{gio:full:talagrand}. The extension to~\ref{FH3} is a consequence of a standard approximation argument.
\end{proof}

\subsection{Convergence over a fixed time-window}

In this section we show in Theorem~\ref{4:wasser turnpike} that the  entropic interpolations for \ref{KSP} and \ref{KFSP} enjoy a turnpike property  with respect to the Wasserstein distance. 

 Notice that a turnpike property in the Wasserstein distance could be deduced from the entropic one (cf.\ Theorem~\ref{teo entropic turnpike} and~\ref{full:teo entropic turnpike}) by means of the Talagrand inequality~\eqref{talagrand}. However, below we provide a different proof that is of independent interest for two reasons. Firstly, the inequality below holds for any $t\in[0,T]$, while the entropic turnpike is restricted to the sub-interval $[\delta,T-\delta]$. Secondly, the argument in the proof, which uses the optimal control formulation of the Schr\"odinger problem, will be instrumental for the study of the short-time behaviour of the Schr\"odinger bridge.
\begin{theorem}[Wasserstein turnpike]\label{4:wasser turnpike}
 Under hypotheses \ref{H1},~\ref{H2} and~\ref{H3}, there exists a positive constant $C_{d,\alpha,\beta,\gamma}$ such that for any $0<\delta\leq1$, as soon as $T>\tempo$, for any $t\in[0,T]$ it holds
 \begin{equation*}
  \cW_2(\mu^T_{t},\,\frm)\leq C_{d,\alpha,\beta,\gamma} \,\delta^{-\frac{3}{2}}\,e^{-\kappa\,[t\wedge(T-t)]}\,\sqrt{\cC_T(\mu,\nu)}\,.
 \end{equation*}
\end{theorem}

\begin{proof} Let us firstly assume \ref{H4}. We we will prove our result for the distorted Wasserstein distance $\cW_{M,2}$ induced by the metrics $\abs{\cdot}_M$. Fix $\delta\in(0,1)$ and assume $t\in[0,T-\delta]$.
Define $\tilde \mu^T_\cdot$ as the marginal flow generated by the uncontrolled process $Z^{\bm 0,T}_s\coloneqq (X^{\bm 0,T}_s,\,V^{\bm 0,T}_s)_{s\in[0,T]}$ solution of \eqref{langevin} started at the initial distribution $\mu^T_0\in\cP(\R^{2d})$.
 Then, since $\tilde \mu^T_0=\mu^T_0$, it holds
\begin{equation}\label{popo}
 \cW_{M,2}(\mu^T_{t},\frm)\leq\cW_{M,2}(\mu^T_{t},\tilde\mu^T_{t})+\cW_{M,2}(\tilde\mu^T_{t},\frm)\overset{\eqref{BEcontW}}{\leq} \cW_{M,2}(\mu^T_{t},\tilde\mu^T_{t})+ e^{-\kappa\,t}\,\cW_{M,2}(\mu^T_0,\frm)\,,
\end{equation}
The second term in the right hand side can be handled with the Talagrand inequality:
\begin{equation*}
 \cW_{M,2}(\mu^T_0,\frm)\lesssim \cW_2(\mu^T_0,\frm)\overset{\eqref{talagrand}}{\lesssim} \sqrt{\cH(\mu^T_0|\frm)}\leq \,\sqrt{\cC_T(\mu,\nu)}\,,
\end{equation*}
where the last step holds since $\cC_T(\mu,\nu)\geq \cH\left((\mathrm{proj}_{x_1})_{\#}\mu^T|(\mathrm{proj}_{x_1})_{\#}\rmR_{0,T}\right)$.

 Let us now focus on $\cW_{M,2}(\mu^T_{t},\tilde\mu^T_{t})$. We will use a synchronous coupling between these two measures. Therefore we introduce
  the process $Z^{\bm u,T}_s\coloneqq (X^{\bm u,T}_s,\,
    V^{\bm u,T}_s)\sim\mu^T_s
$, i.e.\ the solution of \eqref{eq: KSP admissible P} (driven by the same Brownian motion for $Z^{\bm 0,T}_s$) when considering the control ${\bm u}_s=2\gamma\,\nabla_v \log g^T_s(X_s^{\bm u,T},V^{\bm u,T}_s)$. Particularly, from \eqref{pde fg} it follows that $\bm u$ is the optimal control and $Z^{\bm u,T}_s\sim\mu^T_s$. 
For notation's sake set $Z^{\bm \Delta, T}_s\coloneqq Z^{\bm u,T}_s-Z^{\bm 0,T}_s$. Then it holds
\[\De Z^{\bm \Delta,T}_s=\biggl[b\left(Z^{\bm u,T}_s\right)-b\left(Z^{\bm 0,T}_s\right)\biggr]\De s+\begin{pmatrix} 0\\ \bm u_s\end{pmatrix}\De s\,,\]
where $b(z)$ denotes the drift of the Langevin dynamics \eqref{langevin}.
By  It\^o's Formula we obtain
\begin{equation*}
 \begin{aligned}
  \De \abs{Z^{\bm \Delta,T}_s}_M^2&=2MZ^{\bm \Delta}_s\cdot\Bigl(b(Z^{\bm u}_s)-b(Z^{\bm 0}_s)\Bigr)\De s + 2MZ^{\bm \Delta}_s\cdot\begin{pmatrix}0\\ \bm u_s                                                                      \end{pmatrix}\De s\\
&= 2\int_0^1 Z^{\bm \Delta,T}_s\cdot MJ_b\left(rZ^{\bm u,T}_s+(1-r)Z^{\bm 0,T}_s\right)Z^{\bm \Delta,T}_s\De r\,\De s+2MZ^{\bm \Delta,T}_s\cdot\begin{pmatrix}0\\ \bm u_s                                                                      \end{pmatrix}\De s\\
&\leq -2\,\kappa\abs{Z^{\bm \Delta,T}_s}_M^2\De s+2MZ^{\bm \Delta,T}_s\cdot\begin{pmatrix}0\\ \bm u_s                                                                      \end{pmatrix}\De s\,,
 \end{aligned}
\end{equation*}
where the last inequality follows from~\eqref{eq:contR}.
By taking the expectation, and applying H\"older's inequality we get
\begin{equation*}\begin{aligned}\frac{\De}{\De s}\,\Rexpect{\abs{Z^{\bm \Delta,T}_s}_M^2}\leq -2\,\kappa\,\Rexpect{\abs{Z^{\bm \Delta,T}_s}_M^2}+2\Rexpect{\abs{Z^{\bm \Delta,T}_s}_M^2}^{\frac{1}{2}}\Rexpect{\abs{(0,\bm{u_s})^T                                             
}^2_M}^{\frac{1}{2}}\,.
\end{aligned}\end{equation*}
Therefore it holds
\begin{equation*}\frac{\De}{\De s}\sqrt{\Rexpect{{\big|Z^{\bm \Delta,T}_s\big|}_M^2}}\,\leq-\kappa\,\sqrt{\Rexpect{{\big|Z^{\bm \Delta,T}_s\big|}_M^2}} +\Rexpect{\abs{(0,\bm{u_s})^T                                             
}^2_M}^{\frac{1}{2}}\,.
\end{equation*}\
Recalling that the optimal control is given by $\bm u_s=2\gamma\,\nabla_v \log g^T_s(X_s^{\bm u,T},V^{\bm u,T}_s)$ we obtain that
\begin{equation*}\frac{\De}{\De s}\sqrt{\Rexpect{{\big|Z^{\bm \Delta,T}_s\big|}_M^2}}\,\lesssim \left(\int_{\R^{2d}}\abs{\nabla_v \log g^T_s}^2\rho^T_s\,\De \frm\right)^{\frac{1}{2}}\lesssim \psi^T(s)^{\frac{1}{2}}\,.\end{equation*}
Therefore, by integrating over $s\in[0,\,t]$ we get
 \begin{equation*}
  \begin{aligned}
   &\sqrt{\Rexpect{{\big|Z^{\bm \Delta,T}_t\big|}_M^2}}\,=\int_0^{t} \frac{\De}{\De s}\sqrt{\Rexpect{{\big|Z^{\bm \Delta,T}_s\big|}_M^2}}\,\De s \overset{\eqref{halfcorrectors}}{\lesssim} \,\delta^{-\frac{3}{2}}\,e^{-\kappa\,(T-t)}\,\sqrt{\cC_T(\mu,\nu)}\,,
  \end{aligned}
 \end{equation*}
and hence it holds
\begin{equation}\label{eq:short time}\cW_{M,2}(\mu^T_{t},\tilde\mu^T_{t})\lesssim\delta^{-\frac{3}{2}}\,e^{-\kappa\,(T-t)}\,\sqrt{\cC_T(\mu,\nu)}\,.\end{equation}
Then, from \eqref{popo} we deduce that for any $t\in[0,T-\delta]$ it holds
\begin{equation*}
 \begin{aligned}
  \cW_{M,2}(\mu^T_{t},\frm)\lesssim&\,\delta^{-\frac{3}{2}}\,e^{-\kappa\,(T-t)}\,\sqrt{\cC_T(\mu,\nu)}+e^{-\kappa\,t}\,\sqrt{\cC_T(\mu,\nu)}\\
  \lesssim &\,\delta^{-\frac{3}{2}}\,e^{-\kappa\,[t\wedge(T-t)]}\,\sqrt{\cC_T(\mu,\nu)}\,.
 \end{aligned}
\end{equation*}
By considering the contraction along $P^\ast$, the same argument gives us the same bound for $t\in[\delta,T]$ and therefore on the whole domain $[0,T]$.

In order to relax the assumption to \ref{H3}, it is enough to consider once again the approximation of the optimizer (as in the proof of Theorem \ref{teo entropic turnpike}) together with the lower semicontinuity of the Wasserstein distance.
\end{proof}

The previous argument can also be applied in order to prove Theorem \ref{short time}.
\begin{proof}[Proof of Theorem \ref{short time}] At first, let us assume~\ref{H4}. We have
\begin{equation}\label{appo in short}\begin{aligned}
    \cW_2\left(\mu^T_t,\mu^\oo_t\right)\leq&\, \cW_2\left(\mu^T_t,\tilde\mu^T_t\right)+\cW_2\left(\tilde\mu^T_t,\mu^\oo_t\right)\\
    \overset{\eqref{eq:short time},\eqref{BEcontW}}{\lesssim}&\, \delta^{-\frac{3}{2}}\,e^{-\kappa\,(T-t)}\,\sqrt{\cC_T(\mu,\nu)}+e^{-\kappa\, t}\,\cW_2\left(\mu^T_0,\mu\otimes\frm_V\right)\,,
\end{aligned}\end{equation}
where $\tilde\mu^T_\cdot$ is the marginal flow defined in the previous proof, i.e. the flow generated by the uncontrolled process $(X^{\bm 0,T}_s,\,V^{\bm 0,T}_s)_{s\in[0,T]}$  started at the initial distribution $\mu^T_0\in\cP(\R^{2d})$.
Using the inequality
\begin{equation*}
    \cW_2\left(\mu^T_0,\mu\otimes\frm_V\right)^2 \leq \int_{\R^d}\cW_2\left(\mu^T_0(\cdot|x),\frm_V\right)^2 \De\mu(x)\,,
\end{equation*}
applying Talagrand's inequality for $\frm_V$ and using the additivity of relative entropy \cite[Formula A.8]{LeoSch} we obtain
\begin{equation*}
\begin{aligned}
    \cW_2\left(\mu^T_0,\mu\otimes\frm_V\right)^2\leq&\,2\, \int_{\R^d}\cH\left(\mu^T_0(\cdot|x)|\frm_V\right) \De\mu(x)
    =2\,\cH\left(\mu^T_0|\mu\otimes\frm_V\right)\\
     =&\,2\,\cH(\mu^T_0|\frm)-2\int_{\R^{2d}}\log\frac{\De\left(\mu\otimes \frm_V\right)}{\De \frm}\,\De\mu^T_0
     =2\,\cH(\mu^T_0|\frm)-2\int_{\R^{d}}\log\frac{\De\mu}{\De \frm_X}\,\De\mu\\
     =&\,2\,\cH(\mu^T_0|\frm)-2\,\cH(\mu|\frm_X)\overset{\eqref{expmarginals}}{\lesssim}\delta^{-3}\,\cC_T(\mu,\nu)\,e^{-2\kappa\, T}\,.
\end{aligned}\end{equation*}
By combining the above inequalities with \eqref{appo in short} we get our result.

The extension of the result to the weaker~\ref{H3} follows from the same approximating argument discussed in the previous proof.
\end{proof}

With a similar reasoning one can prove that the Wasserstein turnpike holds also for  KFSP under~\ref{FH3}. Notice that since in this setting we fix the whole marginals at time $0$ and $T$, it holds $\bar{\mu}^T_0=\bar \mu$ and $\bar{\mu}^T_T=\bar\nu$ and therefore in this case we do not need a result similar to Theorem \ref{exp marginal entropies}. Therefore we have the following

\begin{theorem}[Wasserstein turnpike]\label{4:full:wasser turnpike}
 Under hypotheses \ref{H1},\ref{H2} and \ref{FH3}, there exists a positive constant $C_{d,\alpha,\beta,\gamma}$ such that for any $0<\delta\leq1$, as soon as $T>\tempo$, for any $t\in[0,T]$ it holds
 \begin{equation*}
  \cW_2(\bar{\mu}^T_{t},\,\frm)\leq C_{d,\alpha,\beta,\gamma} \,\delta^{-\frac{3}{2}}\,e^{-\kappa\,[t\wedge(T-t)]}\,\sqrt{\cC_T^F\left(\bar\mu,\bar\nu\right)}\,.
 \end{equation*}
\end{theorem}

\section{From compact support to finite entropy}\label{sec: approx}
In this section we discuss two types of approximating sequences that we have used in order to extend our main results from \ref{H4} to \ref{H3}. 

In Section~\ref{subsec: approx opt} we deal with the \emph{Approximation of the optimizer} where we are able to prove the convergence of the entropic cost of the approximated problem to the original entropic cost but not the convergence of the associated entropies of the marginals at time $t=0,T$.
On the other hand in Section~\ref{subsec: approx marg}, we investigate the \emph{Approximation of the marginals}, by approximating directly the marginals and consider the associated Sch\"rodinger problems. In this case, we get the convergence of the marginals' relative entropies, but not the one of the entropic cost.
The two aforementioned strategies produce complementary bounds which can be applied together in order to relax the assumptions from \ref{H4} to \ref{H3}.

We will deal exclusively with the approximations and proofs for \ref{KSP} and omit those for \ref{KFSP}, since the latter can be treated in the same way.

 \subsection{Approximating the optimizer}\label{subsec: approx opt}
 
Fix a couple of marginals $\mu,\,\nu\in\cP(\R^{d})$ satisfying \ref{H3}. We already know that there exists a unique minimizer $\mu^{T}\in\Pi_X\left(\mu,\nu\right)$ for \ref{KSP}, with $\rmR_{0,T}$-density given by $\rho^T$. Now consider an increasing sequence of rectangular compact sets $\left(K_n\right)_{n\in\N}$ in $\R^{4d}$ whose union gives the whole space. For each $n\in\N$ define the probability measure $q^T_n$ as the measure whose $\rmR_{0,T}$-density is given by
 \[\hat{\rho}^T_n=\frac{\De q^T_n}{\De \rmR_{0,T}}\coloneqq \left(\rho^{T}\wedge n\right)\,\frac{\IND_{K_n}}{C_n}\,,\]
 where $C_n\coloneqq \int_{K_n} ( \rho^{T}\wedge n)\,\De\rmR_{0,T}$ is the normalising constant. Notice that  $C_n\uparrow 1$ by monotone convergence and  $\hat{\rho}^{T}_n\to \rho^{T}$.
 For convenience, we fix in this section some $\bar n\in\N$ such that $C_n\geq 1/2$ for any $n\geq \bar n$.

 \begin{lemma}\label{n:lemma}
  The following properties hold true.
  \begin{itemize}
   \item[(i)] The marginals $\mu^n\coloneqq(\mathrm{proj}_{x_1})_{\#}q^T_n$ and $\nu^n\coloneqq(\mathrm{proj}_{x_2})_{\#}q^T_n$ satisfy \ref{H4}.
   \item[(ii)] $q^T_n\weakto\mu^T$.
   \item[(iii)] $\cH\left(q^T_n|\rmR_{0,T}\right)\to\cH\left(\mu^{T}|\rmR_{0,T}\right)=\cC_T(\mu,\nu)$.
  \end{itemize}
 \end{lemma}
\begin{proof} We start with i).
Since $q^T_n$ has compact support, so do its marginals $\mu^n,\,\nu^n$. Moreover if $B\subseteq\R^{d}$ is a Borel set, then 
\begin{equation*}\mu^n(B)=q^T_{n}(B\times\R^{3d})\leq \frac{n}{C_n}\,\int_{B\times\R^{3d}}\De\rmR_{0,T}
=\frac{n}{C_n}\,\rmR_{0,T}(B\times\R^{3d})=\frac{n}{C_n}\,\frm_X(B)
\end{equation*}
and therefore $\norm{\De\mu^n/\De\frm_X}_{L^\oo(\frm_X)}\leq \frac{n}{C_n}$. The same reasoning applies also to $\nu^n$.

The weak convergence in (ii) follows from  dominated convergence.

Let us prove point (iii).
Notice that for each $n\geq \bar n$ it holds 
\[\abs{\hat{\rho}^T_{n}\,\log\hat{\rho}^T_{n}}\leq\max\left\{e^{-1},\,(\rho^T C_{\bar n}^{-1})\,\log(\rho^T C_{\bar n}^{-1})\right\}\,,\]
and the above RHS is $\rmR_{0,T}$-integrable since it holds
\begin{equation*}
 \begin{aligned}
  &\int_{\R^{4d}} (\rho^T C_{\bar n}^{-1})\,\log(\rho^T C_{\bar n}^{-1})\,\De\rmR_{0,T}
  =\, \frac{1}{C_{\bar n}}\,\cC_T(\mu,\nu)+\frac{1}{C_{\bar n}}\log\left(\frac{1}{C_{\bar n}}\right)<\oo\,,
 \end{aligned}
\end{equation*}
  which is finite under \ref{H3}. From the Dominated Convergence Theorem we get (iii).
\end{proof}
 
 \begin{prop}\label{full:approx}
  Assume \ref{H1} and \ref{H3} to be true for $\mu,\,\nu\in\cP(\R^d)$. Let $\mu^{T}$ be the unique minimizer in \ref{KSP} with marginals $\mu,\,\nu$. Suppose we are given a sequence $\left(q^T_n\right)_{n\in\N}\subset\cP(\R^{4d})$ such that
  such that
  \begin{itemize}
   \item[(i)] $q^T_n\weakto\mu^T$,
   \item[(ii)] $\cH\left(q^T_n|\rmR_{0,T}\right)\to\cH\left(\mu^T|\rmR_{0,T}\right)$.
  \end{itemize}
Moreover for each $n\in\N$ define the marginals $\mu^n\coloneqq(\mathrm{proj}_{x_1})_{\#}q^T_n$ and $\nu^n\coloneqq(\mathrm{proj}_{x_2})_{\#}q^T_n$.
Then, for each $n\in\N$, there exists a unique minimizer $\mu^T_n\in\Pi_X\left(\mu^n,\nu^n\right)$ in \ref{KSP}  with marginals $\mu^n,\,\nu^n$. Moreover it holds 
\[\mu^T_n\weakto \mu^T\qquad\text{ and }\qquad\cC_T\left(\mu^n,\nu^n\right)\weakto\cC_T\left(\mu,\nu\right)\,.\]
 \end{prop}
 \begin{proof}
  Firstly, \ref{H3} and the convergence of the entropies in the assumptions imply that $\cH(\mu^n|\frm_X)\,,\cH(\nu^n|\frm_X)\leq \cH\left(q^T_n|\rmR_{0,T}\right) <C$ for some positive constant $C$, uniformly in $n\in\N$. Hence \ref{H3} holds also for $\mu^n,\,\nu^n$. This gives the existence and uniqueness of the minimizer in \ref{KSP} with marginals $\mu^n$ and $\nu^n$  for each $n\in\N$.
  
  Then, from \eqref{boundscemocosto} we deduce
\begin{equation*}\begin{aligned}
 \sup_{n\in\N}\cH\left(\mu^T_n|\rmR_{0,T}\right)=\sup_{n\in\N}\cC_T\left(\mu^n,\nu^n\right)\lesssim 1+\sup_{n\in\N}\bigl[\cH(\mu^n|\frm_X)+\cH(\nu^n|\frm_X)\bigr]\lesssim1+2C\,.               
\end{aligned}
\end{equation*}
  Since the relative entropy $\cH(\cdot|\rmR_{0,T})$ has compact level sets \cite[Lemma 1.4.3]{DupuisEllis}, there is a subsequence $\left(\mu^{T}_{n_k}\right)_{k\in\N}$ and a probability measure $\tilde\mu^{T}\in\cP(\R^{4d})$ such that $\mu^{T}_{n_k}\weakto\tilde\mu^{T}$ weakly. Moreover, from the lower semicontinuity of $\cH(\cdot|\rmR_{0,T})$ and the optimality of $\mu_{n_k}^T$ we get
 \begin{equation}\label{full:appoggio}\cH\left(\tilde\mu^{T}|\rmR_{0,T}\right)\leq \liminf_{k\to\oo}\cH\left(\mu^{T}_{n_k}|\rmR_{0,T}\right)\leq \liminf_{k\to\oo}\cH\left(q^{T}_{n_k}|\rmR_{0,T}\right)=\cH(\mu^T|\rmR_{0,T})\,.\end{equation}  
  
  Now, we claim that $\tilde \mu^{T}\in\Pi_X(\mu,\nu)$. Indeed we have for any $i=1,2$
 \begin{equation*}(\mathrm{proj}_{x_i})_{\#}\tilde\mu^{T}=\lim_{k\to\oo}(\mathrm{proj}_{x_i})_{\#}\mu^{T}_{n_k}=\lim_{k\to\oo}(\mathrm{proj}_{x_i})_{\#} q^{T}_{n_k}=(\mathrm{proj}_{x_i})_{\#}\mu^T=\begin{cases}
                                                                                                                                                                                                                                 \mu\quad i=1\,\\
                                                                                                                                                                                                                                 \nu\quad i=2\,,                                                                                                                                                                                                                   \end{cases}
\end{equation*}
where the second equality holds because $\mu^{T}_{n_k}$ and $q^{T}_{n_k}$ share the same marginals, while the third follows from our hypotheses. Therefore, from the bound \eqref{full:appoggio} and the optimality of $\mu^T$ as unique minimizer in $\Pi_X(\mu,\nu)$ for \ref{KSP}, it follows $\tilde\mu^{T}=\mu^T$.

Hence, as $k\to\infty$, it holds $\mu^{T}_{n_k}\weakto \mu^T$ and
\[\exists\,\lim_{k\to\oo}\cC_T\left(\mu^{n_k},\nu^{n_k}\right)=\lim_{k\to\oo} \cH\left(\mu^{T}_{n_k}|\rmR_{0,T}\right)=\cH(\mu^T|\rmR_{0,T})=\cC_T(\mu,\nu)\,.\]
Since in both the limits above the limit objects do not depend on the subsequence and since the weak convergence is metrizable, we get  the desired thesis.
 \end{proof}

 \begin{corollary}\label{full:cor:approx}
  Under the same setting of the previous proposition, if $\mu^{n,T}\in\cP(\Omega)$ denotes the minimizer in \ref{KSPd}, then for each $t\in[0,\,T]$
  \[\mu^{n,T}\weakto\mu^T\qquad\text{ and }\qquad\mu^{n,T}_t\weakto\mu^T_t\,.\]
 \end{corollary}
\begin{proof}
  From the relation between \ref{KSPd} and \ref{KSP}, for any $\phi\in C_b(\Omega)$ we have
 \begin{equation*}
  \int_\Omega \phi\,\De\mu^{n,T}=\int_{\R^{4d}} \left(\int_\Omega \phi\,\De\rmR^{x,v,y,w}\right)\De \mu^T_{n}\to\int_{\R^{4d}} \left(\int_\Omega \phi\,\De\rmR^{x,v,y,w}\right)\De \mu^T=\int_\Omega \phi\,\De \mu^T\,,
 \end{equation*}
 where $\rmR^{x,v,y,w}$ denotes the bridge of the reference measure.
Let us just justify the middle step. Since $\phi$ is bounded, so does $\int_\Omega \phi\,\De\rmR^{x,v,y,w}$. Moreover since the bridge $\rmR^{x,v,y,w}$ is weakly continuous with respect to its extremes \cite[Corollary 1]{markovbridges}, from the continuity of $\phi$, it follows the continuity of the function $(x,v,y,w)\mapsto \int_\Omega \phi\,\De\rmR^{x,v,y,w}$.
Hence the above function is bounded and continuous on $\R^{4d}$ and from the weak convergence $\mu^{T}_n\weakto \mu^{T}$ it follows $\mu^{n,T}\weakto \mu^T$. The other limit  follows by taking the time marginals of $\mu^{n,T}$.
\end{proof}

\subsection{Approximating the marginals}\label{subsec: approx marg}
In this section we are going to perform the approximating arguments directly on the fixed marginals. This will not lead to the convergence of the respective kinetic entropic costs, nevertheless it will be useful in proving the bounds where the previous approximating argument fails. The idea is similar to the one performed previously: consider an increasing sequence of compact sets $\left(K_n\right)_{n\in\N}$ in $\R^{d}$ whose union gives the whole space and for any $\mathfrak{q}\in\cP(\R^d)$, satisfying \ref{H3} and $n\in\N$ large enough so that $\frq(K_n)>0$, define the probability measure $\mathfrak{q}_n^M$ as the measure whose $\frm_X$-density is given by
 \[\frac{\De\mathfrak{q}^M_n}{\De \frm_X}\coloneqq \left(\frac{\De\mathfrak{q}}{\De \frm_X}\wedge n\right)\,\frac{\IND_{K_n}}{C^\mathfrak{q}_n}\,,\]
 where $C_n^{\mathfrak{q}}\coloneqq \int_{K_n} (\frac{\De\mathfrak{q}}{\De \frm_X}\wedge n)\,\De\frm_X \geq 0$ is the normalising constant. Note that monotone convergence yields $C_n^\mathfrak{q}\uparrow 1$. Then it follows that $\mathfrak{q}^M_n$ satisfies \ref{H4}, $\mathfrak{q}^M_n\weakto \mathfrak{q}$ and by mimicking the argument performed in the Lemma \ref{n:lemma} it follows that
 \begin{equation}\label{conv:approx:marginals}
 \cH(\mathfrak{q}_n^M|\frm_X) \overset{n\to\oo}{\longrightarrow} \cH(\mathfrak{q}|\frm_X)\,.
 \end{equation}
 
 \begin{lemma}\label{M:n:lemma}
 Fix $\mu,\nu\in\cP(\R^d)$ satisfying \ref{H3}. Then, up to restricting ourselves to a subsequence,  it holds
 \[\cC_T(\mu,\nu) \leq \liminf_{n\to\oo}\cC_T(\mu^M_n,\nu^M_n)\,.\]
 \end{lemma}
 \begin{proof}
Let $\mu^T_{M,n}$ denotes the optimizer for $\cC_T(\mu^M_n,\nu_n^M)$. Then we have 
\[\cH(\mu^{T}_{M,n}|\rmR_{0,T})=\cC_T(\mu^M_n,\nu_n^M)\overset{\eqref{boundscemocosto}}{\lesssim}1 + \cH(\mu^M_n|\frm_X)+\cH(\nu^M_n|\frm_X)\,\overset{n\to\oo}{\longrightarrow}\,1+ \cH(\mu|\frm_X)+\cH(\nu|\frm_X)\,,\]
which is finite because of \ref{H3}. Since $\cH(\cdot|\rmR_{0,T})$ has compact level set, we know that there exists $\mu^\star\in\cP(\R^{4d})$ such that $\mu_{M,n}^{T}\weakto \mu^\star$, up to considering a subsequence. We claim that $\mu^\star\in\Pi_X(\mu,\nu)$. Indeed we have $ (X_0)_{\#}\mu^T_{M,n}\weakto(X_0)_{\#}\mu^\star$ but $(X_T)_{\#}\mu^T_{M,n}=\mu^M_n\weakto \mu$ and hence $(X_0)_{\#}\mu^\star=\mu$. Similarly it holds $(X_T)_{\#}\mu^\star=\nu$. Therefore we have $\cC_T(\mu,\nu) \leq \cH(\mu^\star|\rmR_{0,T})$ and from the lower semicontinuity of $\cH(\cdot|\rmR_{0,T})$  we deduce our thesis.
 \end{proof}

\appendix
\section{}
\footnotesize

\subsection*{Proof of Lemma~\ref{lemmaappendicebound}}
 
Let $T_0>0$ be fixed. From Jensen's inequality we know that
\begin{equation}\label{appo1appendo}\begin{aligned}\,&\log p_T((x,v),\, (y,w))=\log\int_{\R^{2d}} p_{T-T_0/2}\left((x,v),\,(z,u)\right)\,p_{T_0/2}\left((z,u),\,(y,w)\right)\De \mathfrak{m}(z,u)\\
&\geq\int_{\R^{2d}} \log p_{T-T_0/2}\left((x,v),\,(z,u)\right)\De \mathfrak{m}(z,u)+\int_{\R^{2d}} \log p_{T_0/2}\left((z,u),\,(y,w)\right)\De \mathfrak{m}(z,u)\,.
\end{aligned}\end{equation}
By Theorem 1.1 in \cite{DelarueMenozzi}, there exists  $C\geq 1$ depending on $T_0$ such that 
\begin{equation*}\begin{aligned}p_{T_0/2}\left((z,u),\,(y,w)\right)\gtrsim  p_{T_0/2}((z,u),(y,w))\frm(y,w)
\geq C^{-1} \,e^{-C|\theta_{T_0/2}(z,u)-(y,w)^T|^2}\,\end{aligned}\end{equation*}
where $\theta_t(x_0,v_0)=\left(\theta_t^x,\,\theta^v_t\right)^T$ denotes the solution of the denoised Langevin ODE system
\begin{equation*}\label{nonoiselane}
 \begin{cases}
  \frac{\De}{\De t} \theta_t^x=\theta^v_t\\
  \frac{\De}{\De t} \theta_t^v=-\theta_t^v-\nabla U(\theta_t^x)
 \end{cases}\qquad \text{with}\quad\theta_0=(x_0,\,v_0)^T\,.
\end{equation*}
Since under \ref{H1} there exists a large enough positive $r\in\R$ such that \eqref{eq:contR} holds for $(\Id,-r)$, from \cite[Theorem 1]{monmarche2020almost} (with $\Sigma=0$) it follows
\begin{equation}\label{euro2020}\abs{\theta_t(y,w)-\theta_t(0,0)}\leq  e^{r\,t}\abs{(y,w)^T}\,,\quad\forall\,(y,w)^T\in\R^{2d},\,\,\forall t\geq0\,.\end{equation}
Therefore, up to changing the constants $C$ from line to line, we have
\begin{equation}\label{primoboundlemma}
 \begin{aligned}\int_{\R^{2d}} \log p_{T_0/2}&\left((z,u),\,(y,w)\right)\De \mathfrak{m}(z,u)\geq \log C^{-1} - C \int_{\R^{2d}} \abs{\theta_{T_0/2}(z,u)-(y,w)^T}^2 \De \mathfrak{m}(z,u)\\
  &\geq -C\Bigl(1+\abs{y}^2+\abs{w}^2+\int_{\R^{2d}} \abs{\theta_{T_0/2}(z,u)}^2\De \mathfrak{m}(z,u)\Bigr)\geq -C\Bigl(1+\abs{y}^2+\abs{w}^2\Bigr)\,,
 \end{aligned}
\end{equation}
where the last step holds since $\mathfrak{m}\in\cP_2(\R^{2d})$, and therefore
\begin{equation*}\label{appolemmaappendice}\begin{aligned}
\int_{\R^{2d}} \abs{\theta_{T_0/2}(z,u)}^2\De \mathfrak{m}(z,u)&\leq 2\abs{\theta_{T_0/2}(0,0)}^2+2 \int_{\R^{2d}} \abs{\theta_{T_0/2}(z,u)-\theta_{T_0/2}(0,0)}^2\De \mathfrak{m}(z,u)\\
&\overset{\eqref{euro2020}}{\leq} 2\abs{\theta_{T_0/2}(0,0)}^2+2  \,e^{2r} \int_{\R^{2d}} \left(\abs{z}^2+\abs{u}^2\right)\De \mathfrak{m}(z,u)\leq C\,.
\end{aligned}
\end{equation*}

Now, notice that we can rewrite the first integral of the RHS in \eqref{appo1appendo} as
\begin{equation*}\begin{aligned}\int_{\R^{2d}} \log p_{T-T_0/2}&\left((x,v),\,(z,u)\right)\De \mathfrak{m}(z,u)\\=&\int_{\R^{2d}} \log\biggl[\int_{\R^{2d}} p_{T_0/2}\left((x,v),\,(q,r)\right)\,p_{T-T_0}\left((q,r),\,(z,u)\right)\De\mathfrak{m}(q,r)\biggr]\De \mathfrak{m}(z,u).\end{aligned}\end{equation*}
Because of \eqref{physrever}, we know that $p_{T-T_0/2}\left((q,r),\,(z,u)\right)\De\mathfrak{m}(q,r)$ is a probability measure over $\R^{2d}$ and therefore by Jensen's inequality and Fubini the above displacement can be lower bounded by
\begin{equation*}\begin{aligned}
    &\int_{\R^{2d}} \int_{\R^{2d}} \log \left [p_{T_0/2}\left((x,v),\,(q,r)\right)\right] \,p_{T-T_0}\left((q,r),\,(z,u)\right)\De\mathfrak{m}(q,r)  \De \mathfrak{m}(z,u)   \\
     &=\int_{\R^{2d}}\log p_{T_0/2}\left((x,v),\,(q,r)\right)\De \mathfrak{m}(q,r) \overset{\eqref{physrever}}{=}\int_{\R^{2d}}\log p_{T_0/2}\left((q,-r),\,(x,-v)\right)\De \mathfrak{m}(q,r)\\&=\int_{\R^{2d}}\log p_{T_0/2}\left((q,r),\,(x,-v)\right)\De \mathfrak{m}(q,r)\overset{\eqref{primoboundlemma}}{\geq}   -C\Bigl(1+\abs{x}^2+\abs{v}^2\Bigr)\,.
     \end{aligned}
\end{equation*}
Putting the above lower bound and \eqref{primoboundlemma} into inequality \eqref{appo1appendo}, we get
\[\log p_T\left((x,v),\,(y,w)\right)\geq -c_{T_0}\Bigl(1+\abs{x}^2+\abs{v}^2+\abs{y}^2+\abs{w}^2\Bigr)\,.\]

\bibliographystyle{amsplain}
\bibliography{literature}

\end{document}